\documentclass[reqno,a4paper,11pt]{amsart}
\usepackage{amsmath, amssymb, xcolor}
\usepackage{bm}
\IfFileExists{dsfont.sty}{\usepackage{dsfont}}{}
\providecommand{\mathds}{\mathbb}
\usepackage[mathscr]{eucal}
 \usepackage[top=2.5cm, bottom=2.5cm, left=2.5cm, right=2.5cm]{geometry}

\usepackage[foot]{amsaddr} 
\allowdisplaybreaks
\usepackage{mathtools}

\usepackage[utf8]{inputenc}
\usepackage[T5]{fontenc}

\usepackage[shortlabels]{enumitem}
\setlist{topsep=2pt, itemsep=2pt}

\usepackage{thmtools}

\usepackage[colorlinks=true,linkcolor=teal, citecolor=teal]{hyperref}

\usepackage[nameinlink,capitalize]{cleveref}
\crefname{equation}{}{}

\crefname{theo}{Theorem}{Theorems}
\crefname{coro}{Corollary}{Corollaries}
\crefname{prop}{Proposition}{Propositions}
\crefname{defi}{Definition}{Definitions}
\crefname{lemm}{Lemma}{Lemmas}
\crefname{exam}{Example}{Examples}
\crefname{assum}{Assumption}{Assumptions}

\newtheorem{theo}{Theorem}[section]
\newtheorem{coro}[theo]{Corollary}
\newtheorem{lemm}[theo]{Lemma}
\newtheorem{prop}[theo]{Proposition}

\theoremstyle{definition}
\newtheorem{defi}[theo]{Definition}

\newtheorem{exam}[theo]{Example}
\newtheorem{rema}[theo]{Remark}

\numberwithin{equation}{section}

\newcommand{\bR}{\mathbb R}
\newcommand{\bN}{\mathbb N}

\newcommand{\bZ}{\mathbb Z}

\newcommand{\bE}{\mathbb E}

\newcommand{\bP}{\mathbb P}

\newcommand{\1}{\mathds{1}}

\newcommand{\CUI}{\mathrm{CUI}}
\newcommand{\UI}{\mathrm{UI}}

\newcommand{\cB}{\mathcal B}
\newcommand{\cC}{\mathcal C}

\newcommand{\cF}{\mathcal F}

\newcommand{\cK}{\mathcal K}

\newcommand{\cS}{\mathcal S}
\newcommand{\cP}{\mathcal P}

\newcommand{\cX}{\mathcal X}
\newcommand{\cY}{\mathcal Y}

\newcommand{\Leb}{\mathbf{\lambda}}

\newcommand{\lt}{\left}
\newcommand{\rt}{\right}

\newcommand{\ep}{\varepsilon}
\newcommand{\e}{\mathrm{e}}

\newcommand{\od}{\mathrm{d}}

\newcommand{\Cov}{\mathbb{C}\mathrm{ov}}

\newcommand{\trm}[1]{\textrm{#1}}

\def\b{\big}
\def\B{\Big}
\def\bb{\bigg}
\def\BB{\Bigg}

\def\<{\left<}
\def\>{\right>}

\begin{document}

\title[Mean convergence for random elements indexed in measure spaces]{Mean convergence for Banach space-valued random elements indexed in measure spaces}

\author[Nguyen Thi Kim Sang]{Nguyen Thi Kim Sang$^{1}$}
\address{$^1$Department of Mathematics, Vinh University, Nghe An, Vietnam}
\email{n.kimsangvt@gmail.com}

\author[Nguyen Tran Thuan]{Nguyen Tran Thuan$^{2}$}
\address{$^2$School of Mathematics and Statistics, University of Economics Ho Chi Minh City, Ho Chi Minh City, Vietnam}
\email{thuannt@ueh.edu.vn}



\maketitle

{
\begin{center}
\small \textit{This article is dedicated to Professor {\fontencoding{T5}\selectfont Nguyễn Văn Quảng} on the occasion of his 70th birthday}
\end{center}
}

\begin{abstract}
This article studies mean convergence of Banach space-valued
random elements indexed in a family of finite measure spaces. We derive $L^p$-convergence theorems under (compact) uniform integrability in two regimes: a decaying-index-mass regime and a bounded-index-mass regime, the latter requiring a  new dependence structure which is called diagonal negative dependence for the random elements and expressed via the self-product of the index measure.  We provide examples showing that the conditions to obtain the results are sharp and strictly weaker than related conditions in the literature. As a further illustration for the index measure space framework, a functional law of large numbers in $L^p$ on the space of continuous functions is derived, where the random elements are solutions of stochastic differential equations driven by Brownian motions extracted from a common Brownian sheet.

\bigskip


\smallskip

\end{abstract}


\section{Introduction}\label{sec:intro}

Laws of large numbers (LLNs) for Banach space-valued random elements are a classical topic in probability theory and have attracted sustained attention over the past decades. The establishment of LLNs is often closely related to the geometric structure of the underlying Banach space, see, for example, the pioneering work of Hoffmann-J\o{}rgensen and Pisier \cite{HG76}, the subsequent developments by de Acosta \cite{Ac81}, and numerous follow-up works. When no geometry is imposed on the Banach space, additional structure conditions on the random elements are usually required. The aim of this article is to study LLNs in terms of $L^p$-convergence under suitable uniform integrability assumptions without any geometry of the underlying Banach space.

Recall that a family $\{X_i:i\in I\}$ of random elements taking values in a separable Banach space $(\cX, \|\cdot\|)$ is \textit{uniformly integrable} (UI) if, for every $\ep>0$, there is a closed ball $\bar B_\ep$ centered at the origin such that $\sup_{i\in I}\bE [\|X_i \| \1_{\{X_i\notin \bar B_\ep\}}]\le \ep$. Replacing the closed ball $\bar B_\ep$ by a compact subset $K_\ep\subset\cX$ leads to the stronger notion called \textit{compact uniform integrability} (CUI), see \cite{HG76}. It is obvious that the notions of UI and CUI are equivalent in finite-dimensional spaces.  Over the years, the UI/CUI condition has become a standard assumption to achieve LLNs for Banach space-valued random elements. The existing literature, however, essentially considers the discrete-index setting in which the random elements are indexed in a discrete set and the weighting is encoded by a scalar array $\{a_{n,k}\}$, see, e.g., \cite{BCS14, Ca94, Ca97, CV05, CRUV20, CCRV21, CW12, RT25, TW79, TQ20, TQN14, WR84, WR87} and the references therein.

Recently, the authors in \cite{GGT25} proposed a notion of CUI, which is called $(\nu_\theta)$-$\CUI$, in  a general framework where the random elements are indexed in an arbitrary family of finite measure spaces $\{(S_\theta, \cS_\theta, \nu_\theta) : \theta \in \Theta\}$, which we call in this article the index measure space framework. The role of the measures $\nu_\theta$ is to encode the weighting scheme, and with particular choices of $\nu_\theta$ as discrete measures one can recover the classical, Ces\`aro, and weighted UI/CUI as special cases, see the discussion and related  examples in \cite{GGT25}. Although \cite{GGT25} provided various characterizations of $(\nu_\theta)$-$\CUI$, no limit theorem was derived in this setting.

In this article, we specialize the set $\Theta = \bN$ and consider the family of index measure  spaces $\{(S_n, \cS_n, \nu_n) : n \ge 1\}$. Then, we derive $L^p$-convergences for a family of random elements which satisfies some variants of $(\nu_n)$-CUI/UI conditions. Specifically, our contributions are as follows:
\begin{enumerate}[(a)]
\item \textit{New notions of $(\nu_n)$-$\UI$ and $(\nu_n)$-$\CUI$.} Inspired by the approaches
in \cite{CV05, CRUV20}, we introduce in \cref{def:UI,def:CUI} the notions of $(\nu_n)$-$\UI(p)$ with respect to (w.r.t.) a threshold sequence $\{h(n)\}$ and of $(\nu_n)$-$\CUI(p)$ w.r.t. a sequence $\{K_n\}$ of compact sets. These reduce to the existing UI/CUI conditions by specializing $\nu_n$, $h(n)$, or $K_n$.

\item \textit{Mean convergence under decaying index masses.}
When the total mass $\nu_n(S_n)$ of the index measure decays to zero at a faster rate than the threshold $h(n)$ grows, we prove in \cref{thm:decay-mass} an $L^p$-convergence under $(\nu_n)$-$\UI(p)$ w.r.t. $\{h(n)\}$ without requiring any additional structure of the random elements. As a by-product, this result gives directly the $L^p$-convergence under $(\nu_n)$-$\CUI$ w.r.t. $\{K_n\}$, see \cref{cor:decay-mass-CUI}. These results extend and refine \cite[Theorem 2]{CRUV20} in several aspects, see \cref{rema:simplify-CRUV20}.

\item \textit{Mean convergence under bounded index masses.}
When $\{\nu_n(S_n)\}$ is uniformly bounded, we prove in \cref{thm:bounded-masses} an $L^p$-convergence under $(\nu_n)$-$\CUI(p)$ w.r.t. $\{K_n\}$. The additional structure condition of the random elements is described via the product measure $\nu_n \otimes \nu_n$ in which the  product space $S_n \times S_n$ of index pairs is split into two disjoint regions. The first one is a \textit{near-pair} set, on which any dependence of random elements is allowed. To describe the dependence on the second region, called the \textit{far-pair} set, we introduce the notion of \textit{diagonal negative dependence}. To the best of our knowledge, this dependence structure appears to be new in a general Banach space and it is strictly weaker than (pairwise) independence, see \cref{rema:PDND}. In particular, \cref{thm:bounded-masses} extends \cite[Theorem~1]{CRUV20} to a general index
space, reduces the exponent of the covering factor, and requires no
pairwise independence of the random elements. We also provide in \cref{sec:examples} three examples to illustrate the sharpness of the conditions of \cref{thm:bounded-masses}.

\item \textit{Beyond the discrete-index setting.}
A feature of the index measure space framework is that the index measures $\nu_n$ are not necessarily discrete. As an illustration, we establish in \cref{sec:functional-limit-thm} a functional LLN in $L^p$ where the underlying Banach space is $C([0,1])$, the space of continuous functions on $[0,1]$. The random elements are diffusion processes solving a family of stochastic differential equations driven by Brownian motions induced from a common Brownian sheet. \cref{exam:continuous-nu} provides a situation for such a functional LLN where the index measure $\nu_n$ is absolutely continuous w.r.t. the Lebesgue measure.
\end{enumerate}

The article is organized as follows. In \cref{sec:preliminaries}, we fix the notation, recall the notion of Bochner integral and discuss the notions of (compact) uniform integrability. Main convergence results are contained in \cref{sec:main-results}. \cref{sec:discuss-exam} provides illustrative examples for the main results.

\section{Preliminaries}\label{sec:preliminaries}

\subsection{Notation}

Let $\bN := \{1, 2, \ldots\}$. For a set $S$,
denote by $\cP(S)$ the collection of all subsets of $S$. The complement of a
set $A$ is denoted by $A^c$. For $a, b \in \bR$, we write
$a \vee b := \max\{a, b\}$ and $a \wedge b := \min\{a, b\}$. Let
$\lceil \cdot \rceil$ denote the ceiling function, and $\log$ stands for the
natural logarithm.

Let $(M, \mathcal{M}, \mu)$ be a measure space. For $A \in \mathcal{M}$, we
denote by $\1_A$ the indicator function of $A$. For $p \in (0, \infty)$,
define $L^p(\mu) := L^p(M, \mathcal{M}, \mu)$ as the space of
$\mathcal{M}/\mathcal{B}(\bR)$-measurable functions $f \colon M \to \bR$
satisfying $\int_M |f(x)|^p \, \mu(\od x) < \infty$, and let $L^\infty(\mu)$
be the space of essentially bounded measurable functions. For
$x \in M$, the Dirac measure $\delta_x$ is defined by $\delta_x(A) = 1$ if
$x \in A$ and $\delta_x(A) = 0$ otherwise, for $A \in \mathcal{M}$. Obviously, $\delta_x$ is a probability measure on $\mathcal{M}$. Given measure
spaces $(M_1, \mathcal{M}_1, \mu_1)$ and $(M_2, \mathcal{M}_2, \mu_2)$, their
product measure space is denoted by
$(M_1 \times M_2, \mathcal{M}_1 \otimes \mathcal{M}_2, \mu_1 \otimes \mu_2)$. Let $\Leb$ be the one-dimensional Lebesgue measure.

\subsection{Bochner integral}

Throughout this article, we let $(\mathcal{X}, \|\cdot\|)$ be a separable Banach
space endowed with the Borel $\sigma$-algebra $\mathcal{B}(\mathcal{X})$
induced by the norm.

Assume $(M, \mathcal{M}, \mu)$ is a finite measure space. An
$\mathcal{M}/\mathcal{B}(\mathcal{X})$-measurable mapping 
$f \colon M \to \mathcal{X}$ is \emph{Bochner integrable} w.r.t.
$\mu$ if $\int_M \|f(x)\| \, \mu(\od x) < \infty$, in which case the \textit{Bochner
integral} $\int_M f(x) \, \mu(\od x) \in \mathcal{X}$ is well-defined and
satisfies
$\| \int_M f(x) \, \mu(\od x) \| \le \int_M \|f(x)\| \, \mu(\od x).
$
For further details, we refer the reader to \cite{HNVW16}.

Let $(\Omega, \mathcal{F}, \bP)$ be a  probability space which is assumed to be complete throughout this article. An $\mathcal{F}/\mathcal{B}(\mathcal{X})$-measurable $X \colon \Omega \to \mathcal{X}$ is called an ($\mathcal{X}$-valued)
\emph{random element}. If $\bE\|X\| < \infty$, then the Bochner integral of
$X$ w.r.t. $\bP$ is called the \emph{Bochner expectation} of $X$,
denoted by $\bE X \in \mathcal{X}$. In the case $\mathcal{X} = \bR$, a random
element $X$ is called a random variable. The covariance of two square integrable random variables $X, Y$ is denoted by $\Cov(X, Y) :  =\bE XY - \bE X\, \bE Y$.

Let $\mathcal{K}(\mathcal{X})$ denote the collection
of all nonempty compact subsets of $\mathcal{X}$. For $K \in \mathcal{K}(\mathcal{X})$
and $\varepsilon > 0$, define the metric size $\|K\|$ and the covering number $N(K, \ep)$ of $K$ by setting
\begin{align*}
    \|K\| &:= \max\{\|x\| : x \in K\}, \\
    N(K, \varepsilon) &:= \min\Bigl\{ n \ge 1 : K \text{ is covered by } n
    \text{ open balls of radius } \varepsilon
    \text{ with centers in } K \Bigr\}.
\end{align*}

\subsection{Notions of (compact) uniform integrability} 
Assume from now until the end of this article that  $\{(S_n, \cS_n, \nu_n) : n \ge 1\}$ is a family of finite measure spaces satisfying
\begin{align}\label{eq:finite-spaces}
	\sup_{n \ge 1} \nu_n(S_n) =: C_\nu \in (0, \infty).
\end{align}

Motivated by the approaches in \cite{CV05, CRUV20} and \cite{GGT25}, we propose the following notions.

\begin{defi}\label{def:UI} Let $p \in (0, \infty)$ and assume $\{h(n) : n \ge 1\}$ is a sequence of non-negative real numbers. A collection $\{X^n_s : s\in S_n, n \ge 1\}$ of $\cX$-valued random elements is  \textit{$(\nu_n)$-uniformly $p$-th order integrable w.r.t. $\{h(n)\}$} (or $(\nu_n)$-$\UI(p)$ w.r.t. $\{h(n)\}$, for short) if:
	\begin{enumerate}[\rm (a)]
		\item\label{item:def:UI:measurability} the map $(\omega, s) \mapsto X^n_s(\omega)$ is $\cF \otimes \cS_n/\cB(\cX)$-measurable for every $n \ge 1$,
		
		\item \label{item:def:UI:limit} the following convergence holds:
		\begin{align}\label{eq:UI-condition}
			\lim_{n \to \infty} \bE\int_{S_n} \lt\|X^n_s\rt\|^p \1_{\{\|X^n_s\| > h(n)\}} \nu_n(\od s)  =0.
		\end{align}
	\end{enumerate}
	
\end{defi}

\begin{defi}\label{def:CUI} Let $p \in (0, \infty)$ and assume
 $\{K_n : n \ge 1\}$ is a family of compact subsets of $\cX$. A collection $\{X^n_s : s\in S_n, n \ge 1\}$ of $\cX$-valued random elements is \textit{$(\nu_n)$-compactly uniformly $p$-th order integrable w.r.t. $\{K_n\}$} (or $(\nu_n)$-$\CUI(p)$ w.r.t. $\{K_n\}$, for short) if:
	\begin{enumerate}[\rm (a)]
		\item\label{item:def:CUI:measurability} the map $(\omega, s) \mapsto X^n_s(\omega)$ is $\cF \otimes \cS_n/\cB(\cX)$-measurable for every $n \ge 1$,
		
		\item \label{item:def:CUI:seq} the following convergence holds:
		\begin{align}\label{eq:CUI-condition}
			\lim_{n \to \infty} \bE\int_{S_n} \lt\|X^n_s\rt\|^p \1_{\{X^n_s \notin K_n\}} \nu_n(\od s)  =0.
		\end{align}
	\end{enumerate}	
\end{defi}

\begin{rema}\label{rema:CUI-Kn}
If $\bE \int_{S_n} \|X^n_s\|^p \,\nu_n(\od s) <\infty$ for all $n \ge 1$ then there is a sequence $\{K'_n\}$ of compact sets such that $\bE\int_{S_n} \|X^n_s\|^p \,\1_{\{X^n_s \notin K'_n\}} \nu_n(\od s) \le 1/n$, which implies $(\nu_n)$-$\CUI(p)$ w.r.t. $\{K'_n\}$. This assertion can be argued by noting that $A \mapsto \bE \int_{S_n} \|X^n_s\|^p\, \1_{\{X^n_s \in A\}} \nu_n(\od s)$ is a finite measure on $\cB(\cX)$ and then applying \cite[Theorem 1.3]{Bi99}. The point of \cref{def:CUI} is that the family $\{K_n\}$ is prescribed and should be suitably chosen so that its metric complexity (measured by a covering number and metric size) does not grow too fast in relation to the index measure $\nu_n$ to achieve $L^p$-convergence, see  \cref {cor:decay-mass-CUI} and \cref{thm:bounded-masses} below.
\end{rema}

In the definition of $(\nu_n)$-$\CUI(p)$ w.r.t. $\{K_n\}$, the set $K_n$ may grow as $n \to \infty$ to accommodate more values of $X^n_s$ so that \eqref{eq:CUI-condition} holds. However, if one can choose those compact sets in a uniform way then one arrives at the following version of $\CUI$.

\begin{defi}\label{def:CUI-uniform}
	Let $p \in (0, \infty)$. A collection $\{X^n_s : s\in S_n, n \ge 1\}$ of $\cX$-valued random elements is  \textit{$(\nu_n)$-compactly uniformly $p$-th order integrable} (or $(\nu_n)$-$\CUI(p)$, for short) if:
	\begin{enumerate}[\rm (a)]
		\item\label{item:def:CUI:uniform-measurability} the map $(\omega, s) \mapsto X^n_s(\omega)$ is $\cF \otimes \cS_n/\cB(\cX)$-measurable for every $n \ge 1$,
		
		\item \label{item:def:CUI:uniform}for any $\ep >0$, there exists $K_\ep \in \cK(\cX)$ such that
		\begin{align*}
			 \sup_{n \ge 1} \bE\int_{S_n} \lt\|X^n_s\rt\|^p \1_{\{X^n_s \notin K_\ep\}} \nu_n(\od s)  \le \ep.
		\end{align*}
	\end{enumerate}	
\end{defi}
If the $\nu_n$ are discrete, then the $(\nu_n)$-$\CUI(p)$ boils down to the $\{a_{n, k}\}$-$\CUI(p)$ in \cite{Ca97}. It is clear that if  $\{X^n_s : s\in S_n, n \ge 1\}$ is $(\nu_n)$-$\CUI(p)$ then compact sets $\{K_{n}\}$ exist such that it is $(\nu_n)$-$\CUI(p)$ w.r.t. $\{K_{n}\}$. However, the converse does not hold as shown in the following example.
\begin{exam}
Let $\cX = \ell^1$ be the Banach space of summable real sequences
with canonical basis $(e_k)_{k\ge1}$. For each $n\ge1$ let $S_n=\{1,\dots,n\}$ and $\nu_n(\{k\})=1/n$ the uniform distribution on $S_n$, and set
$X^n_k:=e_k$ for $k\in S_n$. Then,
\begin{align*}
	\bE\int_{S_n}\lt\|X^n_s\rt\|^p_{\ell^1} \nu_n(\od s)
	=\frac1n\sum_{k=1}^{n}\|e_k\|_{\ell^1}^{p}=1,\quad\forall n\ge1.
\end{align*}
Taking $K_n:=\{e_1,\dots,e_n\}$, which is a compact subset of $\ell^1$,
we have $X^n_k\in K_n$ for all $k\in S_n$, so the integrand in
\eqref{eq:CUI-condition} is zero, and hence, $\{X^n_k : k \in S_n, n \ge 1\}$ is
$(\nu_n)$-$\CUI(p)$ w.r.t. $\{K_n\}$.

Let $K$ be an arbitrary compact subset of $\ell^1$. Since $\|e_i-e_j\|_{\ell^1}=2$ for $i\ne j$ and $K$ is compact, it can only contain finitely many of the $e_i$. Hence, for sufficiently large $n$ one has
\begin{align*}
	\bE\int_{S_n}\lt\|X^n_s\rt\|^p_{\ell^1} \1_{\{X^n_s\notin K\}}\nu_n(\od s)
	=\frac{1}{n}\,\#\{i\le n : e_i\notin K\}
	= 1 - \frac{1}{n}\#\{i\le n : e_i\in K\} \to 1
\end{align*}
as $n \to \infty$. Therefore, $\{X^n_k : k \in S_n, n \ge 1\}$ is \textit{not}
$(\nu_n)$-$\CUI(p)$.
\end{exam}

We collect some basic properties for the notions above.
\begin{rema}\label{rema:property} We give some properties of $(\nu_n)$-$\UI(p)$ w.r.t. $\{h(n)\}$ and $(\nu_n)$-$\CUI(p)$ w.r.t. $\{K_n\}$.
	\begin{enumerate}
		\item \label{item:1} According to \cite[Remark 6(1)]{GGT25}, the joint measurability condition in \cref{def:UI,def:CUI,def:CUI-uniform} is satisfied if $S_n$ is countable and $\cS_n: = \cP(S_n)$ for $n \ge 1$.
		
		\item \label{item:2} If $\{X^n_s : s \in S_n, n \ge 1\}$ is $(\nu_n)$-$\CUI(p)$ w.r.t. $\{K_n\}$, then it is $(\nu_n)$-$\UI(p)$ w.r.t. $\{h(n)\}$ for $h(n) : = \|K_n\|$. This is immediate by noting that $\1_{\{\|X^n_s\| > h(n)\}} \le \1_{\{X^n_s \notin K_n\}}$, and thus, \eqref{eq:CUI-condition} implies \eqref{eq:UI-condition}.
		
		\item \label{item:3} Let $\cX$ be finite dimensional. If $\{X^n_s : s \in S_n, n \ge 1\}$ is $(\nu_n)$-$\UI(p)$ w.r.t. $\{h(n)\}$, then it is $(\nu_n)$-$\CUI(p)$ w.r.t. $\{K_n\}$ for $K_n : = \{x \in \cX : \|x\| \le h(n)\}$. This can be argued by noting that the set $\{x \in \cX : \|x\| \le h(n)\}$ is compact.

		\item \label{item:4} If $\{X^n_s : s \in S_n, n \ge 1\}$ is $(\nu_n)$-$\UI(p)$ w.r.t. $\{h(n)\}$ (resp. $(\nu_n)$-$\CUI(p)$ w.r.t. $\{K_n\}$), then it is $(\nu_n)$-$\UI(q)$ w.r.t. $\{h(n)\}$ (resp. $(\nu_n)$-$\CUI(q)$ w.r.t. $\{K_n\}$) for any $0< q \le p$. Indeed, by H\"older's inequality,
		\begin{align*}
			\bE \int_{S_n} \lt\|X^n_s\rt\|^q \1_{\{\|X^n_s\| > h(n)\}} \nu_n(\od s) & \le \nu_n(S_n)^{1 - \frac{q}{p}} \lt( \bE \int_{S_n} \lt\|X^n_s\rt\|^p \1_{\{\|X^n_s\| > h(n)\}} \nu_n(\od s) \rt)^{\frac{q}{p}}\\
			& \le C_\nu^{1 - \frac{q}{p}} \lt( \bE \int_{S_n} \lt\|X^n_s\rt\|^p \1_{\{\|X^n_s\| > h(n)\}} \nu_n(\od s) \rt)^{\frac{q}{p}}\\
			& \to 0 \quad \trm{as } n \to \infty.
		\end{align*}
		The argument for the CUI condition is similar.
	\end{enumerate}
\end{rema}

Several characterizations of $(\nu_n)$-$\CUI(p)$ are provided in \cite[Theorem 13]{GGT25}. We only formulate here the characterization that is used in this article.

\begin{prop}[\cite{GGT25}]\label{prop:charac-CUI}
Let $\{X^n_s : s \in S_n, n\ge 1\}$ be a collection of $\cX$-valued random elements such that $(\omega, s) \mapsto X^n_s(\omega)$ is $\cF\otimes \cS_n/\cB(\cX)$-measurable for all $n \ge 1$. Let $p \in (0, \infty)$. Then, the following assertions are equivalent:
\begin{enumerate}[\rm (i)]
\item $\{X^n_s : s \in S_n, n\ge 1\}$ is $(\nu_n)$-$\CUI(p)$.

\item $\{X^n_s : s \in S_n, n\ge 1\}$ is uniformly tight, which means
\begin{align*}
\forall \ep >0, \, \exists K_\ep \in \cK(\cX)\; :\; \sup_{n \ge 1} (\bP \otimes \nu_n) (X^n_s \notin K_\ep) \le \ep,
\end{align*}
and that 
\begin{align*}
\lim_{a \to \infty} \sup_{n \ge 1} \bE \int_{S_n} \|X^n_s\|^p \, \1_{\{\|X^n_s\| > a\}} \nu_n(\od s) =  0.
\end{align*}
\end{enumerate} 
\end{prop}



\section{Main results}\label{sec:main-results}

We first show that the UI/CUI condition with order $p \in [1, \infty)$ implies the well-definedness of $\int_{S_n} X_s\, \nu_n(\od s)$  for sufficiently large $n$ almost surely.
\begin{lemm}\label{lem:existence-integral}
	Let $\{h(n) : n \ge 1\} \subset [0, \infty)$, $\{K_n : n \ge 1\} \subseteq \cK(\cX)$, and  $p \in [1, \infty)$. Assume that $\{X^n_s : s \in S_n, n \ge 1\}$ is either $(\nu_n)$-$\UI(p)$ w.r.t. $\{h(n)\}$ or  $(\nu_n)$-$\CUI(p)$ w.r.t. $\{K_n\}$. Then, there is $n_0 \in \bN$ such that $\bE \int_{S_n} \lt\|X^n_s\rt\|^p \nu_n(\od s) < \infty$ for all $n \ge n_0$. Consequently,  for any $n \ge n_0$, one has $\bE \int_{S_n} \lt\|X^n_s\rt\| \nu_n(\od s) < \infty$, and thus:
	\begin{enumerate}[\rm (1)]
	\item\label{item:lem:Fubini:1} $s \mapsto \bE X^n_s$ is $\cS_n/\cB(\cX)$-measurable and Bochner $\nu_n$-integrable.
	
	\item\label{item:lem:Fubini:2} $\omega \mapsto \int_{S_n} X^n_s \, \nu_n(\od s)$ is $\cF/\cB(\cX)$-measurable and Bochner $\bP$-integrable.
	
	\item\label{item:lem:Fubini:3} $\int_{S_n} \lt\|X^n_s\rt\| \nu_n(\od s) < \infty$ a.s. and $\bE\|X^n_s\| < \infty$ for $\nu_n$-a.e. $s \in S_n$.
	\end{enumerate}
\end{lemm}

\begin{proof} We only give the proof for the CUI condition, the argument for  UI is analogous. Due to \eqref{eq:CUI-condition}, there exists $n_0 \in \bN$ such that for all $n \ge n_0$,
	\begin{align*}
		\bE\int_{S_n} \lt\|X^n_s\rt\|^p \1_{\{X^n_s \notin K_n\}} \nu_n(\od s) \le 1.
	\end{align*}
Then, for $n \ge n_0$, one has
\begin{align*}
	\bE\int_{S_n} \lt\|X^n_s\rt\|^p  \nu_n(\od s)  & = \bE\int_{S_n} \lt\|X^n_s\rt\|^p \1_{\{X^n_s \in K_n\}} \nu_n(\od s) + \bE\int_{S_n} \lt\|X^n_s\rt\|^p \1_{\{X^n_s \notin K_n\}} \nu_n(\od s) \\
	& \le \|K_n\|^p \nu_n(S_n) + 1 \le \lt\|K_n\rt\|^p C_\nu + 1 <\infty.
\end{align*}
Applying H\"older's inequality and noting that $\nu_n(S_n) <\infty$, we obtain $\bE \int_{S_n} \lt\|X^n_s\rt\| \nu_n(\od s) < \infty$. Items \ref{item:lem:Fubini:1} and \ref{item:lem:Fubini:2}  are direct application of Fubini's theorem for Bochner integrals (see, e.g., \cite[Proposition 1.2.7]{HNVW16}), and item \ref{item:lem:Fubini:3}  is a standard Fubini argument.
\end{proof}

\subsection{Mean convergence under decaying index masses} Our first result gives convergence in $L^p$ under a mass-decaying condition.
\begin{theo}\label{thm:decay-mass}
	Let $\{h(n) : n \ge 1\} \subset [0, \infty)$ and $p \in [1, \infty)$. Assume that $\{X^n_s : s \in S_n, n \ge 1\}$ is $(\nu_n)$-$\UI(p)$ w.r.t. $\{h(n)\}$. If
	\begin{align}\label{eq:thm:decay-mass}
		\lim_{n \rightarrow \infty} h(n) \nu_n(S_n)=0,
	\end{align}
	then
	\begin{align}\label{eq:thm:decay-masses}
	\lim_{n \to \infty} \mathbb{E}  \lt\| \int_{S_n} X^n_s \, \nu_n(\od s) \rt\|^p  = 0,
	\end{align}
	and in particular, $\int_{S_n} X^n_s \, \nu_n(\od s)  \to 0$ in probability  as $n \to \infty$.
\end{theo}

\begin{proof}
For $n_0$ in \cref{lem:existence-integral}, applying the triangle inequality we obtain for all $n \ge n_0$ that
\begin{align*}
	&\lt( \mathbb{E}  \lt\| \int_{S_n} X^n_s \, \nu_n(\od s) \rt\|^p\rt)^{\frac{1}{p}} \\
	& \le \lt( \mathbb{E}  \lt\| \int_{S_n} X^n_s \1_{\{\|X^n_s\| \le h(n)\}} \, \nu_n(\od s) \rt\|^p\rt)^{\frac{1}{p}} + \lt( \mathbb{E}  \lt\| \int_{S_n} X^n_s \1_{\{\|X^n_s\| > h(n)\}} \, \nu_n(\od s) \rt\|^p\rt)^{\frac{1}{p}}\\
	& \le \lt( \mathbb{E}  \lt| \int_{S_n} \|X^n_s\| \1_{\{\|X^n_s\| \le h(n)\}} \, \nu_n(\od s) \rt|^p\rt)^{\frac{1}{p}} + \lt( \mathbb{E}  \lt| \int_{S_n} \|X^n_s\| \1_{\{\|X^n_s\| > h(n)\}} \, \nu_n(\od s) \rt|^p\rt)^{\frac{1}{p}}\\
	& \le h(n)  \nu_n(S_n) + \nu_n(S_n)^{\frac{1}{q}} \lt( \mathbb{E}   \int_{S_n} \lt\|X^n_s\rt\|^p \1_{\{\|X^n_s\| > h(n)\}} \, \nu_n(\od s) \rt)^{\frac{1}{p}}\\
	&\le h(n) \nu_n(S_n) + C_\nu^{\frac{1}{q}} \lt( \mathbb{E}   \int_{S_n} \lt\|X^n_s\rt\|^p \1_{\{\|X^n_s\| > h(n)\}} \, \nu_n(\od s) \rt)^{\frac{1}{p}},
\end{align*}	
where we apply H\"older's inequality with $\frac{1}{q} : = 1 - \frac{1}{p} \in [0, 1)$ for the second term on the right-hand side to get the last estimate. Then, \eqref{eq:thm:decay-masses} immediately follows from \eqref{eq:UI-condition} and \eqref{eq:thm:decay-mass}.
\end{proof}

\begin{coro}\label{cor:decay-mass-CUI}
	Let $\{K_n : n \ge 1\} \subseteq \cK(\cX)$ and $p \in [1, \infty)$. Assume that $\{X^n_s : s \in S_n, n \ge 1\}$ is $(\nu_n)$-$\CUI(p)$ w.r.t. $\{K_n\}$. If
	\begin{align}\label{eq:thm:decay-mass-CUI}
		\lim_{n \rightarrow \infty} \lt\|K_n\rt\| \nu_n(S_n)=0,
	\end{align}
	then \eqref{eq:thm:decay-masses} holds true.
\end{coro}

\begin{proof}
	It immediately follows from \cref{thm:decay-mass} and \cref{rema:property}\eqref{item:2}.
\end{proof}

\begin{coro}\label{cor:decay-mass}
	Let $\{K_n : n \ge 1\} \subseteq \cK(\cX)$ and $p \in [1, \infty)$. Let $\{u_n, v_n : n \ge 1\} \subseteq \bZ\cup\{\pm \infty\}$ with $u_n < v_n$ for all $n\ge 1$ and let $\{a_{n, k} : k \in [u_n, v_n]\cap \bZ, n \ge 1\} \subset \bR$ satisfy 
	$$\sup_{n \ge 1} \sum_{k \in [u_n, v_n] \cap \bZ} |a_{n, k}| <\infty.$$
Assume that $\{X^n_k : k\in S_n, n \ge 1\}$ is $(\nu_n)$-$\CUI(p)$ w.r.t. $\{K_n\}$, where $S_n: = [u_n, v_n] \cap \bZ$, $\cS_n : = \cP(S_n)$, and $\nu_n : = \sum_{k \in S_n} |a_{n, k}| \delta_{k}$. If 
	\begin{align}\label{eq:cor:decay-mass}
		\lim_{n \to \infty} \lt\|K_n\rt\| \sum_{k \in S_n} |a_{n, k}| = 0,
	\end{align}
then
\begin{align*}
	\lim_{n \to \infty} \mathbb{E} \, \BB\| \sum_{k \in S_n} a_{n, k} X^n_k \BB\|^p  = 0.
\end{align*}
\end{coro}

\begin{proof}
Define $\nu_n^{\pm} : = \sum_{k \in S_n} a_{n, k}^{\pm} \delta_{k}$, where $a^+_{n, k} : = \max\{a_{n, k}, 0\}$ and $a^-_{n, k} : = \max\{-a_{n, k}, 0\}$. For $\mu_n \in \{\nu_n^+, \nu_n^-\}$, we have $\mu_n(A) \le \nu_n(A)$ for any $A\in \cS_n$, which verifies that $\{X^n_k : k \in S_n, n \ge 1\}$ is $(\mu_n)$-$\CUI(p)$ w.r.t. $\{K_n\}$, here we use \cref{rema:property}\eqref{item:1} to get the joint measurability. Moreover, \eqref{eq:cor:decay-mass} ensures that $\lt\|K_n\rt\| \mu_n(S_n) \to 0$. Applying \cref{cor:decay-mass-CUI} yields
\begin{align*}
	\lim_{n \to \infty} \mathbb{E}  \, \BB\| \sum_{k \in S_n} a_{n, k}^{\pm} X^n_{k} \BB\|^p  = 0,
\end{align*}
which then completes the proof by $a_{n, k} = a_{n, k}^+ - a_{n, k}^-$ and the triangle inequality.
\end{proof}

\begin{rema}\label{rema:simplify-CRUV20}
	\cref{cor:decay-mass} not only generalizes  but also refines \cite[Theorem 2]{CRUV20}. Specifically, compared with \cite[Theorem 2]{CRUV20}, \cref{cor:decay-mass} (with $p = 1$) does not require the martingale difference condition, and the factor $N(K_n, \ep)$ (which tends to infinity when $K_n$ grows) is not necessarily imposed on the left-hand side of \eqref{eq:cor:decay-mass}. Moreover, our proof provides a simpler route to achieve the desired convergence (the technique in the proof of \cref{thm:decay-mass} can be directly applied to \cite[Theorem 2]{CRUV20}). 
	
	We also note that the absence of the covering factor $N(K_n, \ep)$ is specific to the decaying index mass setting in \cref{cor:decay-mass-CUI}, whereas for bounded index masses in \cref{thm:bounded-masses} below we show in \cref{exam:covering-necessary} that this factor cannot be removed.
\end{rema}

\subsection{Mean convergence under bounded index masses}
For condition \eqref{eq:thm:decay-mass-CUI} in \cref{cor:decay-mass-CUI}, it typically happens that $\|K_n\|$ is bounded in $n$ or $ \|K_n\| \to \infty$ which then enforces $\nu_n(S_n) \to 0$ as $n \to \infty$. Since the requirement $\nu_n(S_n) \to 0$ is too restrictive for applications, in the following we provide a setting in which such a   decaying index mass condition is not necessarily satisfied.

Let us first introduce the following dependence structure for random elements. Note that this notion can be naturally formulated in general metric spaces also.

\begin{defi}[Diagonal negative dependence] \mbox{ }
\begin{enumerate}[(i)]
\item Two $\cX$-valued random elements $X, Y$ are said to be \textit{diagonally negatively dependent} if 
\begin{align}\label{def:PDND}
\bP(X \in A, Y \in A) \le \bP(X \in A)\,\bP(Y \in A), \quad \forall A \in \cB(\cX),
\end{align}
or equivalently,
$\Cov(\1_{\{X \in A\}}, \1_{\{Y \in A\}}) \le 0$ for all $A \in \cB(\cX)$.

\item Let $\varnothing \neq I \subseteq \bR$ and $M \in \bN \cup \{0\}$.
A family $\{Y_i : i \in I\}$ of $\cX$-valued random elements is \textit{$M$-pairwise diagonally negatively dependent} if $Y_i, Y_j$ are diagonally negatively dependent for all $i, j \in I$ with $|i-j|>M$.
When $M=0$, we then simply call $\{Y_i : i \in I\}$ \textit{pairwise diagonally negatively dependent}.

\end{enumerate}
\end{defi}

The following lemma provides some basic properties of diagonal negative dependence.
\begin{lemm}\label{lem:DND-properties}
Let $X, Y$ be $\cX$-valued random elements.
\begin{enumerate}[\rm (i)]
\item If $X$ and $Y$ are independent, then they are diagonally negatively dependent.

\item $X, Y$ are diagonally negatively dependent if and only if $Y, X$ are. Moreover, \eqref{def:PDND} holds for a set $A \in \cB(\cX)$ if and only if it holds for $A^c$.

\item Let $(\cY, \|\cdot\|_\cY)$ be a separable Banach space and let $f \colon \cX \to \cY$ be $\cB(\cX)/\cB(\cY)$-measurable. If $X, Y$ are diagonally negatively dependent, then so are $f(X), f(Y)$.
\end{enumerate}
\end{lemm}

\begin{proof}
The first two assertions are clear, so we only verify (iii). For any $B \in \cB(\cY)$, one has
\begin{align*}
\bP(f(X) \in B, f(Y) \in B) = \bP(X \in f^{-1}(B), Y \in f^{-1}(B)) \le \bP(f(X) \in B) \bP(f(Y) \in B),
\end{align*}
which completes the proof.
\end{proof}

\begin{rema}\label{rema:PDND} 
\begin{enumerate}[(i)]
\item The terminology \textit{diagonal} means that condition \eqref{def:PDND} is only imposed on the diagonal events $\{(X, Y) \in A\times A\}$, $A \in \cB(\cX)$, rather than on the events $\{(X, Y) \in A \times B\}$ of arbitrary products $A\times B$.

\item For every $M \in \bN \cup \{0\}$, the $M$-pairwise dependence of $\{Y_i : i \in I\}$ (i.e., $Y_i$ and  $Y_j$ are independent whenever $|i - j| > M$) clearly implies $M$-pairwise diagonal negative dependence, but the converse does not hold in general. We provide in \cref{exam:sampling-without-replacement} below a family of random elements appearing in the context of sampling without replacement from a finite population that is pairwise diagonally negatively dependent but neither pairwise independent nor a martingale difference sequence.

\item There are several well-studied dependence structures for real-valued random
variables such as  negative orthant/quadrant dependence \cite{Le66} and negative association \cite{JP83}, which crucially rely on the total order of the real line but none of them has a direct analogue in a general (separable) Banach space. Numerous results on mean convergence and laws of large numbers for UI/CUI random elements rely instead on (pairwise) independence or martingale difference, see, e.g., the recent articles \cite{BCS14, CRUV20, CCRV21, RT25} and the references therein. As far as we are aware, diagonal negative dependence appears to be a new dependence structure for random elements in a Banach space without geometric assumptions. Together with a CUI-type condition, this allows us to obtain $L^p$-convergence, $p \in [1, \infty)$.
\end{enumerate}
\end{rema}

\begin{exam}\label{exam:sampling-without-replacement} Let $2\le N \in \bN$ and let $V_N: = \{v_1, \ldots, v_N\}$ be a family of distinct elements in $\cX$. Fix $2\le n \in \bN$ with $n \le N$. We let $(X_1, \ldots, X_n)$ be the first $n$ draws of a uniform sample \textit{without replacement} from the set $V_N$. Denote $l(A) : = \#\{1\le i \le N : v_i \in A\}$ for $A \in \cB(\cX)$. For $i \ne j$, 
since both $X_i$ and $X_j$ are uniformly distributed on $V_N$, we get for any $A \in \cB(\cX)$ that
 \begin{align*}
 \bP(X_i \in A)  \bP(X_j \in A) =  \bP(X_1 \in A)^2 = \frac{l(A)^2}{N^2}. 
 \end{align*}
Note that 
\begin{align*}
\bP(X_i \in A, X_j \in A) = \frac{l(A)(l(A) - 1)}{N(N - 1)},
\end{align*}
from which we get
\begin{align*}
&\bP(X_i \in A, X_j \in A) -  \bP(X_i \in A)  \bP(X_j \in A) \\
& = \frac{l(A)(l(A) - 1)}{N(N - 1)} - \frac{l(A)^2}{N^2}\\
& = - \frac{l(A)}{N} \lt(1 - \frac{l(A)}{N}\rt) \frac{1}{N - 1}\\
& \le 0.
\end{align*}
The strict inequalities hold for those $A$ with $1 \le l(A) \le N -1$. Hence, $\{X_1, \ldots, X_n\}$ are pairwise diagonally negatively dependent but not pairwise independent.

\smallskip

Furthermore, the centered version of $\{X_1, \ldots, X_n\}$ is not a martingale difference in the sense that $\bE[X_{k+1} -\bE X_{k+1} \,|\, X_1, \ldots, X_k] \neq 0$ for any $k = 1, \ldots, n-1$. Indeed, one expresses
\begin{align*}
\bE[X_{k+1} - \bE X_{k+1} \, |\, X_1, \ldots, X_k] = \frac{1}{N-k} \sum_{v \in V_N \backslash \{X_1, \ldots, X_k\}} v - \frac{1}{N} \sum_{i=1}^N v_i,
\end{align*}
 which is not equal to zero in general. To see this, for instance, let $N$ be an even number and let $V_N = \{\pm v_i, i = 1, \ldots, N/2\}$ consist of linearly independent elements $v_i$, $i = 1, \ldots, N/2$. Then,  $\frac{1}{N}\sum_{i=1}^N v_i = 0$. However, for any $k < N$,  it is easy to check that $ \sum_{v \in V_N \backslash \{X_1, \ldots, X_k\}} v$ is non-zero with positive probability, which then yields the desired assertion.
\end{exam}

We are in a position to formulate the main result in this part.

\begin{theo}\label{thm:bounded-masses} Let $\{K_n : n \ge 1\} \subseteq \cK(\cX)$ and let $p \in [1, \infty)$. Assume that $\{X^n_s : s \in S_n, n \ge 1\}$ is $(\nu_n)$-$\CUI(p)$ w.r.t. $\{K_n\}$. Assume for any $n \ge 1$, there exists $D_n \in \cS_n \otimes \cS_n$ such that:
	\begin{enumerate}[\rm (a)]
		\item \label{item:thm:2:b:1} $\lim_{n \to \infty} N(K_n, \ep)^{1 - (\frac{1}{p} \wedge \frac{1}{2})} \lt\|K_n\rt\| ((\nu_n \otimes \nu_n) (D_n))^{\frac{1}{p} \wedge \frac{1}{2}} = 0$ for all $\ep >0$,
		
		\item \label{item:thm:2:b:2} $X^n_s$ and $X^n_t$ are diagonally negatively dependent for $(\nu_n \otimes \nu_n)$-a.e.\ $(s, t) \in (S_n \times S_n) \backslash D_n$.
	\end{enumerate}
	Then, we have 
	\begin{align}\label{eq:thm:bounded-masses}
		\lim_{n \to \infty} \bE \lt\| \int_{S_n} \lt(X^n_s - \bE X^n_s\rt) \nu_n(\od s) \rt\|^p = 0,
	\end{align}
	and in particular, $\int_{S_n} \lt(X^n_s - \bE X^n_s\rt) \nu_n(\od s)  \to 0$ in probability as $n \to \infty$.
\end{theo}

\begin{proof} According to \cref{lem:existence-integral}, the Bochner integral $\int_{S_n} \lt(X^n_s - \bE X^n_s\rt) \nu_n(\od s)$ exists a.s. for any $n \ge n_0$. 
	Let $\ep>0$ be arbitrary. By the $\CUI(p)$-condition, there exists $n_\ep \ge n_0$ such that
	\begin{align*}
		\bE \int_{S_n} \lt\|X^n_s\rt\|^p \1_{\{X^n_s \notin K_n\}} \nu_n(\od s)  \le \ep^p, \quad \forall n \ge n_\ep.
	\end{align*}
For each $n \ge n_\ep$, since $K_n$ is compact, according to the definition of $N(K_n, \ep)$ there exist  a family $\{x^{n,\ep}_i : 1 \le i \le N(K_n, \ep)\} \subseteq K_n$ and a Borel partition $\{A^{n, \ep}_i : 1 \le i \le N(K_n, \ep)\}$ of $K_n$  (see, e.g., \cite[Lemma 1]{CRUV20}) such that 
$$Y^{n,\ep}_s : = \sum_{i=1}^{N(K_n, \ep)} x^{n, \ep}_i \, \1_{\{X^n_s \in A_i^{n, \ep}\}}$$ satisfies
\begin{align*}
	\|X^n_s\1_{\{X^n_s \in K_n\}}  - Y^{n,\ep}_s \| \le \ep \quad \trm{on } \Omega \trm{ for all } s \in S_n.
\end{align*}
Then, for all $n \ge n_\ep$, the triangle inequality yields, a.s., 
\begin{align*}
&\lt\| \int_{S_n} \lt(X^n_s - \bE X^n_s\rt) \nu_n(\od s) \rt\|\\
& \le \lt\| \int_{S_n} \lt(X^n_s \1_{\{X^n_s \notin K_n\}} - \bE X^n_s \1_{\{X^n_s \notin K_n\}}\rt)  \nu_n(\od s) \rt\|\\
& \quad + \lt\| \int_{S_n} \lt((X^n_s \1_{\{X^n_s \in K_n\}} - Y^{n, \ep}_s) - \bE[X^n_s \1_{\{X^n_s \in K_n\}} - Y^{n, \ep}_s]\rt)  \nu_n(\od s) \rt\|\\
& \quad + \lt\| \int_{S_n} (Y^{n,\ep}_s - \bE Y^{n,\ep}_s) \, \nu_n(\od s) \rt\|\\
& \le \int_{S_n} \lt(\|X^n_s\| \1_{\{X^n_s \notin K_n\}} + \bE \|X^n_s\| \1_{\{X^n_s \notin K_n\}}\rt)\nu_n(\od s)\\
& \quad + \int_{S_n} \lt(\|X^n_s \1_{\{X^n_s \in K_n\}} - Y^{n, \ep}_s\| + \bE\|X^n_s \1_{\{X^n_s \in K_n\}} - Y^{n, \ep}_s\|\rt)  \nu_n(\od s) \\
& \quad + \BB\| \sum_{i=1}^{N(K_n, \ep)} x^{n, \ep}_i \int_{S_n} \lt(\1_{\{X^n_s \in A^{n, \ep}_i\}} - \bP(X^n_s \in A^{n, \ep}_i)\rt) \nu_n(\od s) \BB\|\\
& \le  \int_{S_n} \|X^n_s\| \1_{\{X^n_s \notin K_n\}} \nu_n(\od s) + \bE \int_{S_n} \|X^n_s\| \1_{\{X^n_s \notin K_n\}} \nu_n(\od s)  + 2 \nu_n(S_n) \ep \\
& \quad + \sum_{i=1}^{N(K_n, \ep)} \|x_i^{n,\ep}\| \lt| \int_{S_n} \lt(\1_{\{X^n_s \in A^{n, \ep}_i\}} - \bP(X^n_s \in A^{n,\ep}_i) \rt) \nu_n(\od s) \rt|.
\end{align*}
Since $p \in [1, \infty)$ and due to \eqref{eq:finite-spaces}, we apply the triangle inequality and H\"older's inequality with $\frac{1}{q} : = 1 - \frac{1}{p} \in [0, 1)$ to get
\begin{align*}
	& \lt(\bE\lt\| \int_{S_n} \lt(X^n_s - \bE X^n_s\rt)  \nu_n(\od s) \rt\|^p \rt)^{\frac{1}{p}}\\
	& \le 2 \nu_n(S_n)^{\frac{1}{q}} \lt(\bE\int_{S_n} \lt\|X^n_s\rt\|^p \1_{\{X^n_s \notin K_n\}} \nu_n(\od s)\rt)^{\frac{1}{p}} + 2 \nu_n(S_n) \ep\\
	& \quad + \sum_{i=1}^{N(K_n, \ep)} \|x^{n,\ep}_i\| \, \lt(\bE \lt| \int_{S_n} \lt(\1_{\{X^n_s \in A^{n, \ep}_i\}} - \bP(X^n_s \in A^{n, \ep}_i)\rt) \nu_n(\od s) \rt|^p\rt)^{\frac{1}{p}}.
\end{align*}
Note that, for a bounded random variable $\xi$, using Jensen's inequality yields
\begin{align*}
	(\bE|\xi|^p)^{\frac{1}{p}} & \le \begin{cases}
		(\bE|\xi|^2)^{\frac{1}{2}} & \trm{if } p \in [1, 2]\\[5pt]
		\|\xi\|_{L^\infty(\bP)}^{1 - \frac{2}{p}} (\bE |\xi|^2)^{\frac{1}{p}} & \trm{if } p \in (2, \infty) 
	\end{cases} \\
	& = \|\xi\|_{L^\infty(\bP)}^{(1 - \frac{2}{p})\vee 0} (\bE |\xi|^2)^{\frac{1}{p} \wedge \frac{1}{2}}.
\end{align*}
Applying this to $\xi := \int_{S_n} (\1_{\{X^n_s \in A^{n, \ep}_i\}} - \bP(X^n_s \in A^{n, \ep}_i))\, \nu_n(\od s)$ with  $\|\xi\|_{L^\infty(\bP)} \le \nu_n(S_n) \le C_\nu$, and noting that $\|x^{n, \ep}_i\| \le \|K_n\|$, we obtain
\begin{align}\label{thm:eq:bdd:mass:estimate}
	& \lt(\bE\lt\| \int_{S_n} \lt(X^n_s - \bE X^n_s\rt)  \nu_n(\od s) \rt\|^p \rt)^{\frac{1}{p}} \notag \\
	& \le 2 C_\nu^{\frac{1}{q}} \ep + 2 C_\nu \ep \notag \\
	 & \quad + C_\nu^{(1- \frac{2}{p})\vee 0} \|K_n\| \sum_{i=1}^{N(K_n, \ep)}  \bb(\bE \lt| \int_{S_n} \lt(\1_{\{X^n_s \in A^{n, \ep}_i\}} - \bP(X^n_s \in A^{n, \ep}_i)\rt) \nu_n(\od s) \rt|^2 \bb)^{\frac{1}{p} \wedge \frac{1}{2}}.
\end{align}
For $i = 1, \ldots, N(K_n, \ep)$, using Fubini's theorem (twice) yields
\begin{align*}
\trm{T}^{n, \ep}_i &:= \bE \lt| \int_{S_n} \lt(\1_{\{X^n_s \in A^{n, \ep}_i\}} - \bP(X^n_s \in A^{n, \ep}_i)\rt) \nu_n(\od s) \rt|^2  \\
	& = \bE \int_{S_n \times S_n} \lt(\1_{\{X^n_s \in A^{n, \ep}_i\}} - \bP(X^n_s \in A^{n, \ep}_i)\rt) \lt(\1_{\{X^n_t \in A^{n, \ep}_i\}} - \bP(X^n_t \in A^{n, \ep}_i)\rt) (\nu_n \otimes \nu_n)(\od s, \od t)  \\
	& = \int_{S_n \times S_n} \lt( \bP(X^n_s \in A^{n, \ep}_i, X^n_t \in A^{n, \ep}_i) - \bP(X^n_s \in A^{n, \ep}_i) \bP(X^n_t \in A^{n, \ep}_i)\rt) (\nu_n \otimes \nu_n)(\od s, \od t).
\end{align*}
For any $(s, t) \in S_n \times S_n$, noting that
\begin{align*}
R^{n, \ep}_i (s, t): =  \bP(X^n_s \in A^{n, \ep}_i, X^n_t \in A^{n, \ep}_i) - \bP(X^n_s \in A^{n, \ep}_i) \bP(X^n_t \in A^{n, \ep}_i)  \le \bP(X^n_s \in A^{n, \ep}_i),
\end{align*}
and exploiting the fact that $\{A^{n ,\ep}_i : 1 \le i \le N(K_n, \ep)\}$ is a partition of $K_n$, we get
\begin{align*}
\sum_{i = 1}^{N(K_n, \ep)} R^{n, \ep}_i (s, t)  \le \sum_{i = 1}^{N(K_n, \ep)} \bP(X^n_s \in A^{n, \ep}_i) = \bP(X^n_s \in K_n) \le 1.
\end{align*}
We use this estimate and apply condition \ref{item:thm:2:b:2} for $(\nu_n \otimes \nu_n)$-a.e.\ $(s, t) \in D_n^c$  to obtain
\begin{align*}
  \sum_{i = 1}^{N(K_n, \ep)}  \trm{T}^{n, \ep}_i & =   \int_{D_n}  \sum_{i = 1}^{N(K_n, \ep)}   R^{n, \ep}_i (s, t) (\nu_n \otimes \nu_n)(\od s, \od t) \\
	& \quad + \sum_{i = 1}^{N(K_n, \ep)}   \int_{D_n^c} R^{n, \ep}_i (s, t) (\nu_n \otimes \nu_n)(\od s, \od t)\\
	& \le (\nu_n \otimes \nu_n)(D_n).
\end{align*}
Set  $\alpha : = \frac{1}{p} \wedge \frac{1}{2} \in (0, \frac{1}{2}]$  and $\alpha': = 1 - \alpha$. Using H\"older's inequality with $\frac{1}{(1/\alpha)} + \frac{1}{(1/\alpha')} = 1$ yields
\begin{align*}
& \sum_{i=1}^{N(K_n, \ep)}  (\trm{T}^{n, \ep}_i)^{\alpha} \le N(K_n, \ep)^{\alpha'} \BB(  \sum_{i=1}^{N(K_n, \ep)}   \trm{T}^{n, \ep}_i \BB)^{\alpha} \le N(K_n, \ep)^{1- \alpha} ((\nu_n \otimes \nu_n)(D_n))^{\alpha}.
\end{align*}
Plugging those estimates into \eqref{thm:eq:bdd:mass:estimate} gives
\begin{align}\label{eq:thm2:Lp-estimate}
	&\lt(\bE\lt\| \int_{S_n} (X^n_s - \bE X^n_s) \, \nu_n(\od s) \rt\|^p \rt)^{\frac{1}{p}} \notag \\
	& \le 2 C_\nu^{\frac{1}{q}} \ep + 2 C_\nu \ep + C_\nu^{(1- \frac{2}{p})\vee 0} N(K_n, \ep)^{1 - (\frac{1}{p} \wedge \frac{1}{2})} \lt\|K_n\rt\| ((\nu_n \otimes \nu_n)(D_n))^{\frac{1}{p} \wedge \frac{1}{2}}.
\end{align}
Letting $n \to \infty$ and using condition \ref{item:thm:2:b:1} we obtain
\begin{align*}
	\limsup_{n \to \infty} \lt(\bE\lt\| \int_{S_n} (X^n_s - \bE X^n_s) \, \nu_n(\od s) \rt\|^p \rt)^{\frac{1}{p}} \le 2 C_\nu^{\frac{1}{q}} \ep + 2 C_\nu \ep.
\end{align*}
Since $\ep>0$ is arbitrary, \eqref{eq:thm:bounded-masses} follows.
\end{proof}

\begin{rema}\label{rema:cond-b} Let us comment on the conditions of \cref{thm:bounded-masses}.
\begin{enumerate}[(1)]
\item The set $D_n$ collects the near-pairs $(s, t)$ on which any dependence of $X^n_s, X^n_t$ is allowed. However,   the mass of $D_n$ under $\nu_n\otimes \nu_n$ is required to decay at a rate governed by the factor $N(K_n,\varepsilon)\|K_n\|$ measuring a kind of metric complexity of $K_n$. A natural candidate for $D_n$  is a measurable neighborhood of the diagonal $ \{(s,s) : s \in S_n\}$.   In contrast, for far-pairs $(s,t)\in D_n^c$, one requires the diagonal negative dependence of $X^n_s$, $X^n_t$. This condition is satisfied when $X^n_s$, $X^n_t$ are independent for all $(s, t) \in D_n^c$, see \cref{lem:DND-properties}(i).

\item \cref{exam:covering-necessary} shows that the covering factor $N(K_n, \ep)$ in condition \ref{item:thm:2:b:1} cannot be omitted.

\item \cref{exam:sharp-exponent} indicates that, when $p \in [1, 2]$, one cannot reduce the exponent $1 - (\frac{1}{p} \wedge \frac{1}{2}) = \frac{1}{2}$ in the covering factor $N(K_n, \ep)$. It seems to us that the exponent  when $p>2$ might not be optimal, but we are not aware of a way to improve it with the current approach.

\item  \cref{exam:negative-covariance} illustrates condition~\ref{item:thm:2:b:2} in the case of strictly negative covariances.
\end{enumerate}
\end{rema}

\begin{coro}\label{cor:bounded-mass}
Let the assumptions of \cref{cor:decay-mass}, excluding \eqref{eq:cor:decay-mass}, be satisfied. Let $\{M_n : n \ge 1\} \subseteq \bN\cup\{0\}$ be a non-decreasing sequence and assume that, for each $n \ge 1$, the  family $\{X^n_k :  k \in S_n\}$ is $M_n$-pairwise diagonally negatively dependent. If
\begin{align}\label{eq:cor:bounded-mass}
\lim_{n \to \infty} N(K_n, \ep)^{1 - (\frac{1}{p} \wedge \frac{1}{2})} \lt\|K_n\rt\| \BB((1+M_n) \sum_{k\in S_n} a_{n, k}^2\BB)^{\frac{1}{p} \wedge \frac{1}{2}} = 0\quad \trm{for all }  \ep >0,
\end{align}
then
	\begin{align*}
		\lim_{n \to \infty} \mathbb{E}  \,\BB\| \sum_{k \in S_n} a_{n, k} (X^n_k - \bE X^n_k) \BB\|^p  = 0.
	\end{align*}
\end{coro}

\begin{proof} We use the notations as in the proof of \cref{cor:decay-mass}. 
	Define 
	$$D_n : = \{(k, l) \in S_n \times S_n: |k - l| \le M_n\} \in \cS_n \otimes \cS_n.$$
	 For $\mu_n \in \{\nu_n^+, \nu_n^-\}$, by letting $a_{n, j} : =0$ if $j \notin S_n$ we get
	\begin{align*}
		(\mu_n \otimes \mu_n)(D_n) & \le \sum_{k, l \in S_n: \;|k - l| \le M_n} \lt|a_{n, k}\rt| \lt| a_{n, l}\rt| = \sum_{k \in S_n} a_{n, k}^2 + 2 \sum_{j = 1}^{M_n} \; \sum_{k \in S_n} \lt|a_{n, k}\rt| \lt| a_{n, k +j}\rt|\\
		& \le  \sum_{k \in S_n} a_{n, k}^2 + \sum_{j = 1}^{M_n} \; \sum_{k \in S_n} (a_{n, k}^2 + a_{n, k +j}^2) \le (1 + 2M_n) \sum_{k \in S_n} a_{n, k}^2.
	\end{align*}
It then follows from \eqref{eq:cor:bounded-mass} that $$\lim_{n \to \infty} N(K_n, \ep)^{1- (\frac{1}{p} \wedge \frac{1}{2})} \lt\|K_n\rt\| ((\mu_n \otimes \mu_n)(D_n))^{\frac{1}{p} \wedge \frac{1}{2}} = 0.$$
 The $M_n$-pairwise diagonal negative dependence of $\{X^n_k :  k \in S_n\}$ ensures condition \ref{item:thm:2:b:2} to hold for all $k, l \in S_n$ with $|k-l| >M_n$, i.e. $(k, l) \in (S_n \times S_n) \backslash D_n$. As $\{X^n_k : k \in S_n, n \ge 1\}$ is $(\mu_n)$-$\CUI(p)$ w.r.t. $\{K_n\}$, we apply \cref{thm:bounded-masses} to obtain
 \begin{align*}
 	\lim_{n \to \infty} \bE \, \BB\| \sum_{k \in S_n} a_{n, k}^{\pm} (X^n_k - \bE X^n_k) \BB\|^p = 0,
 \end{align*}
 which then verifies the desired conclusion.
\end{proof}

\begin{rema} In particular, when $p = 1$, $M_n = 0$ and the pairwise diagonal negative dependence is strengthened to the pairwise independence for all $n\ge 1$, 
\cref{cor:bounded-mass} still improves  \cite[Theorem 1]{CRUV20} in the sense that the exponent of the covering factor $N(K_n, \ep)$ here is reduced to $1/2$ instead of $1$ as in  \cite[Theorem 1]{CRUV20}.
\end{rema}

When $\{X^n_s : s\in S_n, n \ge 1\}$ satisfies the $(\nu_n)$-$\CUI(p)$ condition in the sense of  \cref{def:CUI-uniform}, then the factor $N(K_n, \ep)\|K_n\|$ in \cref{thm:bounded-masses}\ref{item:thm:2:b:1}  can be omitted.

\begin{prop}\label{thm:CUI-uniform}
Let $p \in [1, \infty)$ and assume that $\{X^n_s : s \in S_n, n \ge 1\}$ is $(\nu_n)$-$\CUI(p)$ in the sense of  \cref{def:CUI-uniform}. Assume that for each $n \ge 1$ there exists $D_n \in \cS_n \otimes \cS_n$ such that:
\begin{enumerate}[\rm (a')]
		\item \label{item:thm:CUI-uni:b:1} $\lim_{n \to \infty} (\nu_n \otimes \nu_n) (D_n) = 0$,
		
		\item \label{item:thm:CUI-uni:b:2} $X^n_s$ and $X^n_t$ are diagonally negatively dependent for $(\nu_n \otimes \nu_n)$-a.e.\ $(s, t) \in (S_n \times S_n) \backslash D_n$.
	\end{enumerate}
	Then, \cref{eq:thm:bounded-masses} holds.
\end{prop}

\begin{proof}
For $\ep >0$, there exists $K_\ep \in \cK(\cX)$ such that $\sup_{n \ge 1} \bE \int_{S_n} \|X^n_s\|^p \, \1_{\{X^n_s \notin K_\ep\}} \nu_n(\od s) \le \ep^p$. We repeat the proof of \cref{thm:bounded-masses} with $K_n$ replaced by the fixed compact set $K_\ep$ so that \eqref{eq:thm2:Lp-estimate} becomes
\begin{align*}
	&\lt(\bE\lt\| \int_{S_n} (X^n_s - \bE X^n_s) \, \nu_n(\od s) \rt\|^p \rt)^{\frac{1}{p}}  \\
	& \le 2 C_\nu^{\frac{1}{q}} \ep + 2 C_\nu \ep + C_\nu^{(1- \frac{2}{p})\vee 0} N(K_\ep, \ep)^{1 - (\frac{1}{p} \wedge \frac{1}{2})} \lt\|K_\ep\rt\| ((\nu_n \otimes \nu_n)(D_n))^{\frac{1}{p} \wedge \frac{1}{2}}.
\end{align*}
Letting $n \to \infty$ and using \ref{item:thm:CUI-uni:b:1}, and then letting $\ep \downarrow 0$, we obtain the convergence \cref{eq:thm:bounded-masses}.
\end{proof}

\section{Discussion and examples}\label{sec:discuss-exam}

\subsection{Examples in the setting of discrete index measures}\label{sec:examples}

In the decaying index mass setting in \cref{cor:decay-mass-CUI} or under the uniform $(\nu_n)$-$\CUI(p)$ in \cref{thm:CUI-uniform}, the covering number $N(K_n, \ep)$ does not enter the mass condition \eqref{eq:thm:decay-mass-CUI} and condition \ref{item:thm:CUI-uni:b:1}, respectively. However, this is not the case for \cref{thm:bounded-masses} and the factor $N(K_n, \ep)$ appearing in condition \ref{item:thm:2:b:1} is necessary. We now show in \cref{exam:covering-necessary} that replacing condition \ref{item:thm:2:b:1} by the weaker condition
\begin{align}\label{eq:weakened-a}
	\lim_{n \to \infty} \lt\|K_n\rt\|  ((\nu_n \otimes \nu_n)(D_n))^{\frac{1}{p} \wedge \frac{1}{2}} = 0 \tag{w-a}
\end{align}
(i.e., omitting the factor $N(K_n, \ep)$) is \emph{not} sufficient to obtain \eqref{eq:thm:bounded-masses}, even under pairwise independence in place of pairwise diagonal negative dependence.

\begin{exam}\label{exam:covering-necessary}
Let $\cX = L^1([0,1], \Leb)=: L^1(\Leb)$, with $\Leb$ the Lebesgue measure, and let $p \in [1, \infty)$ arbitrarily. For each $n \ge 1$ we define
$$f_k^{(n)} := n\, \1_{[\frac{k-1}{n}, \frac{k}{n})}, \quad k = 1, \ldots, n,
$$
so that $\|f_k^{(n)}\|_{L^1(\Leb)} = 1$ and $f_i^{(n)} f_j^{(n)} =0$ for $i \neq j$. Let $(\xi_k)_{k \ge 1}$ be i.i.d. random variables with $\bP(\xi_k = \pm 1) = \frac{1}{2}$, and define
$$X^n_{k} := \xi_k f_k^{(n)}.$$
Let $\bR^n$ be the Banach space equipped with the norm $\|x\|_1: = \sum_{k=1}^n |x_k|$ for $x = (x_1, \ldots, x_n) \in \bR^n$. We define the linear map $T_n \colon (\bR^n, \|\cdot\|_1)  \to L^1([0,1], \Leb)$ by setting  
$$T_n(x): =  \sum_{k=1}^n x_k f_k^{(n)}, \quad x = (x_1, \ldots, x_n).$$
Since $f_i^{(n)} f_j^{(n)} =0$ for $i \neq j$, it holds that
\begin{align*}
\|T_n(x)\|_{L^1(\Leb)} = \BB\| \sum_{k=1}^n x_k f_k^{(n)} \BB\|_{L^1(\Leb)} = \sum_{k=1}^n |x_k| = \|x\|_1,
\end{align*}
which gives the isometry of $T_n$. Denoting by $\bar B^n_1$ the closed unit ball of $(\bR^n, \|\cdot\|_1)$ and defining
$$K_n := T_n(\bar B^n_1) = \lt\{\sum_{k=1}^n x_k f_k^{(n)} : \sum_{k=1}^n |x_k| \le 1\rt\},$$
we find that $K_n$ is a compact subset of $L^1(\Leb)$ due to the isometry of $T_n$ and the compactness of $\bar B^n_1$. Moreover, $\|K_n\|_{L^1(\Leb)} = \sup_{f \in K_n}\|f\|_{L^1(\Leb)} = 1$.

We now show that all conditions of \cref{thm:bounded-masses}, except \ref{item:thm:2:b:1}, are satisfied for the selection
$$S_n := \{1, \ldots, n\}, \quad \cS_n := \cP(S_n), \quad \nu_n := \frac{1}{n} \sum_{k=1}^n \delta_{k}, \quad D_n := \{(k, k) : k \in S_n\}.$$
 Moreover,  \eqref{eq:weakened-a} is also satisfied but the conclusion \eqref{eq:thm:bounded-masses} fails.
 \smallskip
 
(1) \emph{$\{X^n_k : k \in S_n, n \ge 1\}$ is $(\nu_n)$-$\CUI(p)$ w.r.t.\ $\{K_n\}$.} Note that $X^n_k \in \{\pm f_k^{(n)}\} \subset K_n$ due to $\xi_k = \pm 1$, one has $\1_{\{X^n_k \notin K_n\}} = 0$ for all $(n, k)$, and hence, \eqref{eq:CUI-condition} holds trivially.

\smallskip

(2) \emph{Condition \ref{item:thm:2:b:2} holds.} For any $n$ and any $k \ne l$, the independence of $\xi_k$ and $\xi_l$, together with \cref{lem:DND-properties}(i), yield condition \ref{item:thm:2:b:2}.
\smallskip

(3) \emph{Condition \eqref{eq:weakened-a} holds.} Since $(\nu_n \otimes \nu_n)(D_n) = 1/n$ and $\|K_n\|_{L^1(\Leb)} = 1$, we get 
	$$\|K_n\|_{L^1(\Leb)} ((\nu_n \otimes \nu_n)(D_n))^{\frac{1}{p} \wedge \frac{1}{2}} = 1/n^{\frac{1}{p} \wedge \frac{1}{2}} \to 0 \quad \trm{as } n \to \infty.$$

\smallskip

(4) \emph{Condition \ref{item:thm:2:b:1} fails.} Fix $\ep \in (0,1)$. As $K_n$ is isometric to $\bar B^n_1$, we have $N(K_n, \ep) = N(\bar B_1^n, \ep)$. By definition, one can cover $\bar B^n_1$ by $N(\bar B^n_1, \ep)$ open balls $B^n_\ep(x_i)$ with center $x_i \in \bar B^n_1$ and radius $\ep$. Let $\Leb_n$ denote the Lebesgue measure in $\bR^n$, it holds that
	$$\Leb_n(\bar B^n_1) \le \Leb_n \B(\cup_{i = 1}^{N(\bar B^n_1, \ep)} B^n_\ep(x_i) \B) \le N(\bar B^n_1, \ep) \Leb_n( \ep \bar B^n_1) = \ep^n N(\bar B^n_1, \ep)\Leb_n(\bar B^n_1),$$
	 which yields 
	 $$N(\bar B^n_1, \ep) \ge \ep^{-n}.$$
	 Therefore,  as $n \to \infty$,
$$N(K_n, \ep)^{1- (\frac{1}{p} \wedge \frac{1}{2})} \, \|K_n\|_{L^1(\Leb)} ((\nu_n \otimes \nu_n)(D_n))^{\frac{1}{p} \wedge \frac{1}{2}} \ge  \ep^{-n(1- (\frac{1}{p} \wedge \frac{1}{2}))}/ n^{\frac{1}{p} \wedge \frac{1}{2}} \to \infty.$$

\smallskip

(5)  \emph{Conclusion \eqref{eq:thm:bounded-masses} fails.} As $\bE X^n_k= 0$ and $f_i^{(n)} f_j^{(n)} = 0$ for $i \neq j$, one has
$$\lt\| \int_{S_n}(X^n_s - \bE X^n_s)\, \nu_n(\od s) \rt\|_{L^1(\Leb)} = \bb\| \frac{1}{n} \sum_{k=1}^n \xi_k f_k^{(n)} \bb\|_{L^1(\Leb)} = \frac{1}{n} \sum_{k=1}^n |\xi_k|= 1\quad \trm{a.s.},$$
	hence $\bE \,\b\| \int_{S_n}(X^n_s - \bE X^n_s)\, \nu_n(\od s) \b\|_{L^1(\Leb)}^p = 1 \not\to 0$.
\end{exam}

We next show that the exponent $1 - (\frac{1}{p} \wedge \frac{1}{2}) = \frac{1}{2}$ in \cref{thm:bounded-masses}\ref{item:thm:2:b:1} is sharp when $p \in [1, 2]$. Specifically, there exists $\{X^n_s : s \in S_n, n \ge 1\}$ of $(\nu_n)$-$\CUI(p)$ w.r.t. $\{K_n\}$ such that condition   \ref{item:thm:2:b:2} holds, condition \ref{item:thm:2:b:1} is satisfied with the factor $N(K_n, \ep)^\kappa$ for any $\kappa \in (0, \frac{1}{2})$ but fails for $\kappa = \frac{1}{2}$, and the conclusion also fails to hold.

\begin{exam}\label{exam:sharp-exponent}
Let $\cX = \ell^1$ be the Banach space of summable real sequences $x = (x_k)_{k \ge 1}$ with the norm $\|x\|_{\ell^1} = \sum_{k = 1}^\infty |x_k|$. Let $(e_k)_{k \ge 1}$ be the canonical basis of $\ell^1$. We let $p \in [1, 2]$ so that $1 - (\frac{1}{p} \wedge \frac{1}{2}) = \frac{1}{2}$. Take $0 < \kappa < \frac{1}{2}$  arbitrary. Set
\begin{align*}
S_n : = \{1, \ldots, n\}, \quad \cS_n : = \cP(S_n), \quad \nu_n : = \frac{1}{n} \sum_{k = 1}^n \delta_k, \quad D_n :  = \{(k, k) : k \in S_n\}.
\end{align*}
Let $\{\xi_k : k \ge 1\}$ be i.i.d. random variables with $\bP(\xi_k = \pm 1) = \frac{1}{2}$. Define 
$$X^n_k : = \xi_k e_k, \quad  k \in S_n, n \ge 1.$$
Obviously, condition \ref{item:thm:2:b:2} holds due to the independence. It is also clear that $\{X^n_k : k \in S_n, n\ge 1\}$ is $(\nu_n)$-$\CUI(p)$ w.r.t. $\{K^*_n\}$ where $K_n^* := \{\pm e_k : k \in S_n\}$ so that $X^n_k \in K_n^*$ for all $k \in S_n$, $n \ge 1$.  Let us fix $\ep \in (0, \frac{1}{2})$. Since $\|e_i - e_j\|_{\ell^1} = 2$ for $i \ne j$ and $\|e_i - (-e_i)\|_{\ell^1} = 2$, we get 
$$N(K^*_n, \ep) = 2n.$$
Assume $\{X^n_k : k \in S_n, n\ge 1\}$ is $(\nu_n)$-$\CUI(p)$ w.r.t. $\{K_n\}$ for some $\{K_n\} \subseteq \cK(\ell^1)$. Note that $\{X^n_k : k \in S_n, n \ge 1\}$ is also $(\nu_n)$-$\CUI(p)$ w.r.t. $\{K_n \cap K_n^*\}$. As we should choose those compact sets as small as possible (in the sense that the factor $N(K_n, \ep) \|K_n\|$ is small), we may assume that $K_n \subseteq K_n^*$ for all $n \ge 1$.
For any $\ep \in (0, \frac{1}{2})$, there exists $n_\ep$ such that for all $n \ge n_\ep$,
\begin{align*}
\ep \ge \bE \int_{S_n} \|X^n_s\|_{\ell^1}^p \1_{\{X^n_s \notin K_n\}} \nu_n(\od s) = \frac{1}{2n} \sum_{k=1}^n (\1_{\{e_k \notin K_n\}} + \1_{\{-e_k \notin K_n\}}).
\end{align*}
Hence, for any $n \ge n_\ep$,
\begin{align*}
\frac{1}{2n} \sum_{k=1}^n (\1_{\{e_k \in K_n\}} + \1_{\{-e_k \in K_n\}}) \ge 1 - \ep,
\end{align*}
which implies that
\begin{align*}
\# \{k \in [1, n] : e_k \in K_n\} \ge n(1 - 2\ep), \quad n \ge n_\ep.
\end{align*}
As a consequence, $\|K_n\|_{\ell^1} =  1$. Since $\ep \in (0, \frac{1}{2})$, one needs at least $\lceil 2n(1 -\ep) \rceil$ open balls with centers in $K_n$ and radius $\ep$ to cover $K_n$. Thus,
$$N(K_n, \ep) \ge  n (1 - 2\ep), \quad n \ge n_\ep.$$
Combining the arguments above we obtain
\begin{align*}
&N(K_n, \ep)^{\frac{1}{2}} \|K_n\|_{\ell^1} ((\nu_n \otimes \nu_n)(D_n))^{\frac{1}{2}}  \ge (n (1- 2 \ep))^{\frac{1}{2}} n^{-\frac{1}{2}} = \sqrt{1 - 2\ep}
\end{align*}
which does not converge to $0$ when $n \to \infty$, and thus, condition \ref{item:thm:2:b:1} fails to hold.   However,
\begin{align*}
&N(K_n, \ep)^{\kappa} \|K_n\|_{\ell^1} ((\nu_n \otimes \nu_n)(D_n))^{\frac{1}{2}}  \le N(K_n^*, \ep)^{\kappa} \|K_n^*\|_{\ell^1} n^{-\frac{1}{2}} \le 2^\kappa n^{\kappa  - \frac{1}{2}} \xrightarrow{n \to \infty} 0,
\end{align*}
which verifies condition \ref{item:thm:2:b:1} with the factor $N(K_n, \ep)^\kappa$ instead of $N(K_n, \ep)^\frac{1}{2}$.

Finally, since $\bE X^n_k = 0$ we get
$$\bE \lt\|\int_{S_n} (X^n_s - \bE X^n_s) \nu_n(\od s)\rt\|^p_{\ell^1} = \bE\lt\|\frac{1}{n} \sum_{k=1}^n \xi_k e_k \rt\|^p_{\ell^1} = 1 \not\to 0,$$
and hence, the conclusion fails to hold.
\end{exam}

In \cref{exam:negative-covariance} below, we illustrate \cref{thm:bounded-masses} with \emph{strictly} negative covariance condition.

\begin{exam}\label{exam:negative-covariance}
 Let $\cX = \ell^1$ and $(e_k)_{k \ge 1}$ the canonical basis of $\ell^1$. Fix $p \in [1, \infty)$. We assume for each $n \ge 1$ that:
 \begin{itemize}
 \item $j_n, J_n \in \bN$ with $j_n < J_n$ and $J_n \ge n$.
 
 \item A sequence  $\{c_{n, k} : k\ge 1\}$  of positive numbers that satisfies $c_{n, k} \to 0$ as $k \to \infty$. 
  
  \item  For $c^*_n : = \sup_{k \ge 1} c_{n, k}$, one has for all $\ep >0$ that
 $$\BB(1 + \sum_{k \,: \, c_{n, k} \ge \ep} c_{n, k}\BB)^{1- (\frac{1}{p} \wedge \frac{1}{2})} \frac{c^*_n}{n^{\frac{1}{p} \wedge \frac{1}{2}}} \xrightarrow{n \to \infty} 0.$$
 
 \item A sequence  $\{C_{n, k} : j_n + 1 \le k \le J_n\}$ of positive numbers  that satisfies $c^*_n < C_{n, k}$ for all $j_n + 1 \le k \le J_n$ and that
 $$\frac{1}{J_n} \sum_{k = j_n +1}^{J_n} C_{n, k}^p \xrightarrow{n \to \infty} 0.$$
 \end{itemize}
 Define $\{x_{n, k} : 1 \le k \le J_n, n \ge 1\} \subset \ell^1$ by setting
 \begin{align*}
 x_{n, k} : = \begin{cases}
 c_{n, k} \, e_k & \trm{if } 1 \le k \le j_n,\\
 C_{n, k} \, e_k & \trm{if } j_n < k \le J_n.
 \end{cases}
 \end{align*}
 For each $n \ge 1$, we let $(X^n_1, \ldots, X^n_n)$ be the first $n$ draws of a uniform sample \textit{without replacement} from the set $\{x_{n, 1}, \ldots, x_{n, J_n}\}$ of distinct elements. We now verify the conditions in \cref{thm:bounded-masses} for the family $\{X^n_k : k \in S_n, n \ge 1\}$ with the selection
 \begin{align*}
& S_n: = \{1, \ldots, n\}, \quad \cS_n : = \cP(S_n), \quad \nu_n : = \frac{1}{n} \sum_{k=1}^n \delta_k, \quad D_n : = \{(k, k) : k \in S_n\},\\
& K_n : = \cup_{k = 1}^\infty L_{n, k} \quad \trm{where } L_{n, k}: = \{ r e_k : 0 \le r \le c_{n, k}\}.
 \end{align*}
\smallskip

(1) \textit{$K_n$ is compact}: Let $(y^m)_{m \ge 1} \subseteq K_n$, $y^m = (y^m_k)_{k \ge 1} \in \ell^1$. If there exists a $k_0 \in \bN$ such that $L_{n, k_0}$ contains a subsequence of $(y^m)_{m \ge 1}$, then, by the compactness of $[0, c_{n, k_0}]$ in $\bR$, we can find a further subsequence that converges to a limit in $K_n$. Otherwise, the set $\{l(m) : m \ge 1\}$ is infinite where $l(m)$ is the smallest number such that $y^m \in L_{n, l(m)}$. We can find a subsequence $\{k_m : m \ge 1\} \subseteq \bN$ such that $l(k_m) \to \infty$. Since $\|y^{k_m}\|_{\ell^1} \le c_{n, l(k_m)} \to 0 $ as $l(k_m) \to \infty$ by assumption, it implies $y^{k_m} \to 0$ in $\ell^1$. Hence, in both cases we can find a convergent subsequence of $(y^m)_{m \ge 1}$ whose limit belongs to $K_n$. Hence, we obtain the compactness of $K_n$. Moreover, 
 $$\|K_n\|_{\ell^1} = \sup_{x \in K_n} \|x\|_{\ell_1} = \sup_{k \ge 1} c_{n, k} = c^*_n.$$
 
 \smallskip
 
 (2) \textit{$\{X^n_s : s \in S_n, n \ge 1\}$ is $(\nu_n)$-$\CUI(p)$ w.r.t. $\{K_n\}$}: For each $n\ge 1$, observe that $x_{n, k} \notin K_n$ if and only if $j_n +1 \le k \le J_n$. Since $X^n_1, \ldots, X^n_n$ are identically  distributed (but not pairwise independent) with $\bP(X^n_1 = x_{n, k}) = 1/J_n$ for $k = 1, \ldots, J_n$, it implies that
 \begin{align*}
\bE \int_{S_n} \|X^n_s\|_{\ell^1}^p \1_{\{X^n_s \notin K_n\}} \, \nu_n(\od s)  = \bE [\|X^n_1\|_{\ell^1}^p \1_{\{X^n_1 \notin K_n\}}] = \frac{1}{J_n} \sum_{k = j_n +1}^{J_n} C_{n, k}^p \xrightarrow{n \to \infty} 0
 \end{align*}
 by assumption. Hence, the $\CUI(p)$-condition holds.
 
 \smallskip
 
 (3) \textit{Condition \ref{item:thm:2:b:1}}: Let $\ep >0$ and fix $n \ge 1$. Observe that one can cover the set $\cup_{k \, : \, c_{n, k} < \ep} L_{n, k}$ by the open ball center at zero with radius $\ep$. Since $c_{n, k} \to 0$ as $k \to \infty$, there are finitely many $k$ such that $c_{n, k} \ge \ep$.  For those $k$ with $c_{n, k} \ge \ep$, one needs at most $\lceil c_{n, k}/\ep\rceil + 1$ open balls with centers in $K_n$ and radius $\ep$ to cover $L_{n, k}$. Hence,
 \begin{align*}
 N(K_n, \ep) \le 1 + \sum_{k \,: \, c_{n, k} \ge \ep} \lt( \lt\lceil \frac{c_{n, k}}{\ep} \rt\rceil + 1 \rt) \le  1 + 3 \sum_{k \,: \, c_{n, k} \ge \ep} \frac{c_{n, k}}{\ep} \le \lt(1\vee \frac{3}{\ep}\rt) \lt( 1 + \sum_{k \,: \, c_{n, k} \ge \ep} c_{n, k}\rt),
 \end{align*}
 where $\sum_{k \in \varnothing}: = 0$. Since $ \|K_n\|_{\ell^1}  = c^*_n$ and $(\nu_n \otimes \nu_n)(D_n) = 1/n$,  we arrive at
 \begin{align*}
 &N(K_n, \ep)^{1- (\frac{1}{p} \wedge \frac{1}{2})} \|K_n\|_{\ell^1} ((\nu_n \otimes \nu_n)(D_n))^{\frac{1}{p} \wedge \frac{1}{2}}\\
 & \le  \lt(1\vee \frac{3}{\ep}\rt)^{1- (\frac{1}{p} \wedge \frac{1}{2})}  \BB(1 +  \sum_{k \,: \, c_{n, k} \ge \ep} c_{n, k}\BB)^{1- (\frac{1}{p} \wedge \frac{1}{2})} \frac{c^*_n}{n^{\frac{1}{p} \wedge \frac{1}{2}}} \xrightarrow{n \to \infty} 0
 \end{align*}
 by assumption.
 
\smallskip

(4) \textit{Condition \ref{item:thm:2:b:2}}: Fix $n \ge 1$ and let $(i, j) \notin D_n$, i.e. $i \neq j$. By the same arguments as in \cref{exam:sampling-without-replacement} (with $N = J_n$), we find that $X^n_i$ and $X^n_j$ are diagonally negatively dependent and the covariance $\Cov(\1_{\{X^n_i \in A\}}, \1_{\{X^n_j \in A\}})$ is strictly negative for those Borel $A$ with $\#\{1 \le k \le J_n  : x_{n, k} \in A\} \in [1, J_n - 1]$.

 \end{exam}

\begin{rema}\label{rema:enlarging-Kn}
The size of $K_n$ can be increasing in $n$ in the sense that $c^*_n = \|K_n\|_{\ell^1} \to \infty$.
For instance, fix $p \in [1, \infty)$, set $\alpha : = \frac{1}{p} \wedge \frac{1}{2} >0$ and choose
\begin{align*}
c_{n, k} = n^{\alpha/4}/k, \qquad j_n = n^3, \qquad J_n = n^3 + n, \qquad
C_{n, k} = (n+1)^\alpha.
\end{align*}
Then, $c^*_n = n^{\alpha/4}$ and $C_{n, k} > c^*_n$ clearly. Since $\alpha p\le 1$, it holds that
$$\frac{1}{J_n} \sum_{k = j_n + 1}^{J_n} C_{n, k}^p = \frac{n(n+1)^{\alpha p}}{n^3 + n} \xrightarrow{n \to \infty} 0.$$
For any $\ep >0$, one has
$$\BB(1 + \sum_{k \,: \, c_{n, k} \ge \ep} c_{n, k}\BB)^{1- (\frac{1}{p} \wedge \frac{1}{2})} \frac{c^*_n}{n^{\frac{1}{p} \wedge \frac{1}{2}}} = \BB(1 + n^{\alpha/4} \sum_{1 \le k \le n^{\alpha/4} /\ep} \frac{1}{k}\BB)^{1- (\frac{1}{p} \wedge \frac{1}{2})} \frac{n^{\alpha/4}}{n^\alpha} \xrightarrow{n \to \infty} 0$$
where we use the fact that $\sum_{1 \le k \le N} \frac{1}{k} \le 1 + \log N$. Hence, all conditions in \cref{exam:negative-covariance} hold.
\end{rema}

\subsection{A functional law of large numbers in $L^p$}\label{sec:functional-limit-thm}
In this part we illustrate the index measure space framework where the index measures $\nu_n$ are not necessarily discrete and can be  absolutely continuous w.r.t. the Lebesgue measure. The underlying separable Banach space we consider here is the path space of diffusion processes.

Let $\cX = C([0, 1])=: \cC$ be the Banach space of continuous functions $x\colon [0, 1] \to \bR$ equipped with the sup norm $\|x\|_\infty : = \sup_{u \in [0, 1]} |x(u)|$. Note that $(\cC, \|\cdot\|_\infty)$ is separable. Let $\cB(\cC)$ denote the Borel $\sigma$-algebra induced by $\|\cdot\|_\infty$.

Let $\beta  = (\beta(t))_{t \in [0, 1]}$ be a standard Brownian motion, i.e. $\beta(0) = 0$,  $\beta$ has independent and stationary increments, $\beta(t)$ is Gaussian with zero mean and variance $t$, and the path $t \mapsto \beta(\omega, t)$ is continuous on $[0, 1]$ for all $\omega \in \Omega$. Then, $\beta$ can be seen as an $\cF/\cB(\cC)$-measurable map $\beta \colon \Omega \to \cC$ with $\beta(\omega)(\cdot) : = \beta(\omega, \cdot)$ (e.g., by \cite[Remark 2.4.22]{KS91} and the continuity of the restriction map $[0, \infty) \ni t \mapsto t \in [0, 1]$). We refer the reader to \cite{KS91} for relevant properties of Brownian motion and the notion of strong solution of  stochastic differential equations (SDEs) driven by Brownian motion. The following result is well-known.

\begin{lemm}\label{lem:SDE-wellposed}
Let $b\colon [0, 1] \times \bR \to \bR$ and $\sigma \colon [0, 1] \times \bR \to \bR$ be $\cB([0, 1])\otimes \cB(\bR)/\cB(\bR)$-measurable and assume there is a constant $L$ depending only on $(b, \sigma)$ such that, for all $x, y \in \bR$ and $t \in [0, 1]$,
\begin{align*}
& |b(t, x) - b(t, y)| + |\sigma(t, x) - \sigma(t, y)| \le L |x - y|,\\
& |b(t, x)| + |\sigma(t, x)| \le L(1 + |x|).
\end{align*} 
Then, the following SDE with given $x_0 \in \bR$,
\begin{align*}
Y(t) = x_0 + \int_0^t b(u, Y(u)) \od u + \int_0^t \sigma(u, Y(u)) \od \beta(u), \quad t \in [0, 1],
\end{align*}
has a strong solution $Y = (Y(t))_{t \in [0, 1]}$, which is unique up to a  $\bP$-null set, and satisfies
\begin{align*}
\bE\lt[\sup_{0 \le t \le 1} |Y(t)|^q\rt] < \infty, \quad \forall q \in [1, \infty).
\end{align*}
Moreover, there exists a $\cB(\cC)/\cB(\cC)$-measurable $\Phi\colon \cC \to \cC$ depending only on $(x_0, b, \sigma)$ such that
\begin{align*}
Y(\omega, \cdot) = \Phi(\beta(\omega, \cdot)) \quad \trm{for }\bP\trm{-a.s. } \omega \in \Omega.
\end{align*}
\end{lemm}

\begin{proof}
The existence and uniqueness of the strong solution follows from
\cite[Theorem~5.2.9]{KS91}, and the moment estimate from
\cite[Problem~5.3.15]{KS91}. The functional representation $Y = \Phi(\beta)$ with $\Phi\colon \cC \to \cC$ Borel measurable follows from
\cite[Theorem~1]{Ka96} with noting that the
restriction map from $[0, \infty)$ to $[0, 1]$ is continuous.
\end{proof}

Recall that a \textit{Brownian sheet} $W = (W(s, t))_{(s, t) \in [0, \infty) \times [0, 1]}$ is a centered Gaussian process with covariance 
\begin{align*}
\bE[W(s, t) W(s', t')] = (s \wedge s') (t \wedge t'),
\end{align*}
which admits a version with $(s, t) \mapsto W(s, t)$ continuous on $[0, \infty) \times [0, 1]$ almost surely (see \cite[Chapter 1]{Wa86} for its construction and properties). By setting to be zero on a $\bP$-null set, we may assume that all paths of $W$ are continuous, and we always work with such a continuous version in this part.

\begin{prop}\label{prop:continuos-nu}
Let $W = (W(s, t))_{(s, t) \in [0, \infty) \times [0, 1]}$ be a Brownian sheet. Assume that:
\begin{itemize}
\item $\{r_n : n \ge 1\}$ is a sequence of positive numbers with $r_n \to 0$ as $n \to \infty$.
\item $S_n : = [0, \infty)$, $\nu_n$ is a finite measure on $\cS_n: = \cB([0, \infty))$ with $\sup_{n \ge 1} \nu_n(S_n) = C_\nu \in (0, \infty)$ and satisfies
\begin{align*}
(\nu_n\otimes \nu_n)(D_n) \xrightarrow{n \to \infty} 0, \quad \trm{where } D_n : = \{(s, s') \in [0, \infty) \times [0, \infty) : |s - s'| \le r_n\}.
\end{align*}
\item The coefficients $b, \sigma$ satisfy the conditions in \cref{lem:SDE-wellposed}.
\item $x_0 \in \bR$ is fixed. 
\end{itemize}
For each $n \ge 1$ and each $s \in [0, \infty)$, we define
\begin{align*}
B^n_s(t) : = r^{-1/2}_n (W(s+ r_n, t) - W(s, t)), \quad t \in [0, 1].
\end{align*}
Then, $B^n_s = (B^n_s(t))_{t \in [0, 1]}$ is a standard Brownian motion. Let $X^n_s = (X^n_s(t))_{t \in [0, 1]}$ be a strong solution to the SDE
\begin{align}\label{eq:prop:SDE}
X^n_s(t) = x_0 + \int_0^t b(u, X^n_s(u))\od u + \int_0^t \sigma(u, X^n_s(u)) \od B^n_s(u), \quad t \in [0, 1].
\end{align}
Then, there is a process $\widetilde X^n_s = (\widetilde X^n_s(t))_{t \in [0, 1]}$ such that:
\begin{itemize}
\item $\widetilde X^n_s(t) = X^n_s(t)$ for all $t \in [0, 1]$ a.s.,
\item $(\omega, s) \mapsto \widetilde X^n_s(\omega)$ is $\cF\otimes \cB([0, \infty))/\cB(\cC)$-measurable,
\end{itemize}
and for any $p \in [1, \infty)$ it holds
\begin{align}\label{prop:eq:Lp-limit-continuous-nu}
\lim_{n \to \infty} \bE \lt\| \int_0^\infty \lt(\widetilde X^n_s - \bE \widetilde X^n_s \rt)\nu_n(\od s) \rt\|^p_\infty = 0.
\end{align}
\end{prop}

\begin{proof} 

\textbf{Step 1.} We show that  $B^n_s = (B^n_s(t))_{t \in [0, 1]}$ is a standard Brownian motion  for every $n, s$. Indeed, $B^n_s$ is centered Gaussian due to the Gaussianity of $W$. From the covariance structure of $W$, we get, for any $t, t' \in [0, 1]$,
\begin{align*}
\bE[B^n_s(t) B^n_s(t')] & = r_n^{-1} \bE\lt[(W(s+ r_n, t) - W(s, t)) (W(s+ r_n, t') - W(s, t'))\rt]\\
& = r^{-1}_n r_n (t \wedge t') = t\wedge t'.
\end{align*}
Since paths of $B^n_s$ are continuous on $[0, 1]$, which is inherited from that of $W$, we infer that $B^n_s$ is a standard Brownian motion.

By the regularity of the coefficients $b, \sigma$, for each $n \ge 1$ and $s \in [0, \infty)$ there exists a unique strong solution $X^n_s$ to the SDE \eqref{eq:prop:SDE}.  Let $\Phi \colon \cC \to \cC$ be the $\cB(\cC)/\cB(\cC)$-measurable function given in \cref{lem:SDE-wellposed}. Recall that $\Phi$ only depends on $(x_0, b, \sigma)$. Then, 
\begin{align*}
X^n_s(\omega, \cdot) = \Phi(B^n_s(\omega, \cdot)) \quad \trm{for } \bP\trm{-a.s. } \omega \in \Omega.
\end{align*}
Note that the $\bP$-null set here does depend on $s$. We define
\begin{align*}
\widetilde X^n_s(\omega, \cdot ) := \Phi(B^n_s(\omega, \cdot)) \quad \trm{for all } \omega \in \Omega.
\end{align*}
Then, $\widetilde X^n_s$ is a version of $X^n_s$ in the sense that
\begin{align*}
\widetilde X^n_s(\omega, \cdot) = X^n_s(\omega, \cdot) \quad \trm{for } \bP\trm{-a.s. } \omega \in \Omega.
\end{align*}

\noindent \textbf{Step 2.} We verify the conditions of \cref{thm:CUI-uniform} for $\{\widetilde X^n_s : s \in S_n, n \ge 1\}$ with $D_n$ as above.

\smallskip

(1) \textit{Joint measurability}: Fix $n \ge 1$. Since $(\omega, s, t) \mapsto W(\omega, s, t)$ is $\cF \otimes \cB([0, \infty)) \otimes \cB([0, 1])/\cB(\bR)$-measurable and the shift $(s, t) \mapsto (s + r_n, t)$ is continuous, we infer that $(\omega, s, t) \mapsto B^n_s(\omega, t)$ is measurable. In addition, since $t \mapsto B^n_s(\omega, t)$ is continuous for all $(\omega, s)$, it follows from \cite[p. 84]{Bi99} that $(\omega, s) \mapsto B^n_s(\omega, \cdot)$ is $\cF\otimes \cB([0, \infty))/\cB(\cC)$-measurable. Composing with the $\cB(\cC)/\cB(\cC)$-measurable $\Phi$, the map $(\omega, s) \mapsto \widetilde X^n_s(\omega, \cdot) = \Phi(B^n_s(\omega, \cdot))$ is $\cF\otimes \cB([0, \infty))/\cB(\cC)$-measurable.

\smallskip

(2) \textit{$\{\widetilde X^n_s : s \in [0, \infty), n \ge 1\}$ is $(\nu_n)$-$\CUI(p)$ in the sense of \cref{def:CUI-uniform} for any $p \in [1, \infty)$}: Let us fix $p \in [1, \infty)$. Since each $B^n_s$ is a standard Brownian motion and $\Phi$ only depends on $(x_0, b, \sigma)$, we infer that $\widetilde X^n_s = \Phi(B^n_s)$ has the same law on $\cC$ for every $n \ge 1, s \in [0, \infty)$. In other words, the family $\{\widetilde X^n_s : n \ge 1, s \in [0, \infty)\}$ is identically distributed with common distribution denoted by $\mu$. Moreover, a change of variables and \cref{lem:SDE-wellposed} give
\begin{align}\label{eq:moment-bdd-Lq}
\int_{\cC} \|x\|_\infty^q \,\mu(\od x) = \bE \|\widetilde X^n_s\|_\infty^q = \bE \lt[\sup_{0 \le t \le 1} |\widetilde X^n_s(t)|^q \rt] < \infty, \quad \forall q \in [1, \infty).
\end{align}
Since $\mu$ is a probability measure on $\cB(\cC)$ and $(\cC, \|\cdot\|_\infty)$ is complete and separable, it follows from  \cite[Theorem 1.3]{Bi99} that $\mu$ is tight. Hence, for any $\ep>0$, there exists a compact subset $K_\ep$ of $\cC$ such that
\begin{align*}
\bP(\widetilde X^n_s \notin K_\ep) = \mu(K_\ep^c) \le \ep/C_\nu, \quad \forall n \ge 1, s \in [0, \infty).
\end{align*}
Applying Fubini's theorem we get
\begin{align*}
\sup_{n \ge 1} (\bP \otimes \nu_n)(\{(\omega, s) \in \Omega \times S_n : \widetilde X^n_s(\omega) \notin K_\ep\}) & = \sup_{n \ge 1} \int_{S_n} \bP(\widetilde X^n_s \notin K_\ep) \nu_n(\od s) \le \ep,
\end{align*}
which implies the uniform tightness of $\{\widetilde X^n_s : n \ge 1, s \in [0, \infty)\}$. We   apply H\"older's inequality and Fubini's theorem, together with \eqref{eq:moment-bdd-Lq}, to get
\begin{align*}
& \limsup_{a \to \infty} \; \sup_{n \ge 1} \bE \int_{S_n} \|\widetilde X^n_s\|^p_\infty \, \1_{\{\|\widetilde X^n_s\|_\infty > a\}} \nu_n(\od s) \\
& \le  \limsup_{a \to \infty} \; \sup_{n \ge 1} \sqrt{\bE \int_{S_n} \| \widetilde X^n_s \|^{2p}_\infty \, \nu_n(\od s)} \sqrt{\bE \int_{S_n} \1_{\{\|\widetilde X^n_s\|_\infty > a\}} \nu_n(\od s)}\\
& =  \limsup_{a \to \infty} \; \sup_{n \ge 1} \sqrt{\nu_n(S_n) \int_{\cC} \|x\|_\infty^{2p} \, \mu(\od x)}  \sqrt{ \nu_n(S_n) \mu(\{\|x\|_\infty >a\})}\\
& = 0,
\end{align*}
where we use the fact that $\mu$ is a probability measure on $\cB(\cC)$. Hence, it follows from \cref{prop:charac-CUI} that $\{\widetilde X^n_s : n \ge 1, s \in [0, \infty)\}$ is $(\nu_n)$-$\CUI(p)$.
 
 \smallskip

(3) \textit{Condition \ref{item:thm:CUI-uni:b:1}}: This is exactly the assumption $(\nu_n \otimes \nu_n)(D_n) \to 0$ as $n \to \infty$.

\smallskip

(4) \textit{Condition \ref{item:thm:CUI-uni:b:2}}: For all $(s, s') \notin D_n$, i.e. $|s - s'| >r_n$, the intervals $(s, s+ r_n]$ and $(s', s' + r_n]$ are disjoint, which implies the independence of $W(s + r_n, \cdot) - W(s, \cdot)$ and $W(s' + r_n, \cdot) - W(s', \cdot)$. Hence, the two Brownian motions $B^n_s$ and $B^n_{s'}$ are independent, and in particular diagonally negatively dependent by \cref{lem:DND-properties}(i). Since $\widetilde X^n_\cdot = \Phi(B^n_\cdot)$ with $\Phi$ Borel measurable, \cref{lem:DND-properties}(iii) yields that $\widetilde X^n_s$ and $\widetilde X^n_{s'}$ are diagonally negatively dependent, which verifies condition \ref{item:thm:CUI-uni:b:2}.

\smallskip

Finally, applying \cref{thm:CUI-uniform} yields \eqref{prop:eq:Lp-limit-continuous-nu} as desired.
\end{proof}

\begin{exam}\label{exam:continuous-nu} A concrete choice for continuous $\nu_n$ satisfying the conditions of \cref{prop:continuos-nu} is the exponential distribution
$\nu_n(\od s)= \gamma_n \e^{-\gamma_n s} \1_{(0, \infty)}(s) \od s$ with $\gamma_n >0$. Then, $\sup_{n \ge 1} \nu_n ([0, \infty)) = 1$ obviously. By some standard calculations via Laplace distribution we get
\begin{align*}
(\nu_n\otimes \nu_n)(D_n) = \int_{s, s' > 0 \,:\, |s - s'| \le r_n} \gamma_n^2 \e^{-\gamma_n(s + s')} \od s \od s' = 1 - \e^{-\gamma_n r_n},
\end{align*}
which converges to $0$ whenever $\gamma_n r_n \to 0$ as $n \to \infty$.
\end{exam}



\bibliographystyle{spbasic}

\end{document}